\documentclass{tran-l}
\usepackage[breaklinks=true]{hyperref}
\usepackage{amssymb,amsmath,amsthm,mathrsfs,enumerate,tocvsec2}
\usepackage{pgf,tikz}
\usetikzlibrary{arrows}

\newcommand{\bk}{\ensuremath{\Bbbk}}
\newcommand{\scrm}{\ensuremath{\mathscr{M}}}
\newcommand{\m}{\ensuremath{\mathfrak{m}}}
\newcommand{\B}{\ensuremath{\mathcal{B}}}
\newcommand{\T}{\ensuremath{\mathcal{T}}}
\newcommand{\tangle}[1]{\ensuremath{\langle #1 \rangle}}
\newcommand{\commen}[1]{}

\DeclareMathOperator{\Prim}{\ensuremath{Prim}}
\DeclareMathOperator{\Rep}{\ensuremath{Rep}}
\DeclareMathOperator{\Sym}{\ensuremath{Sym}}
\DeclareMathOperator{\Lin}{\ensuremath{Lin}}
\DeclareMathOperator{\Tran}{\ensuremath{Tran}}
\DeclareMathOperator{\id}{\ensuremath{id}}

\theoremstyle{plain}
\newtheorem{theorem}[equation]{Theorem}
\newtheorem{cor}[equation]{Corollary}
\newtheorem{lemma}[equation]{Lemma}
\newtheorem{prop}[equation]{Proposition}

\newtheorem{utheorem}{\textrm{\textbf{Theorem}}}

\theoremstyle{definition}
\newtheorem{definition}[equation]{Definition}
\newtheorem{remark}[equation]{\normalfont{\textit{Remark}}}

\numberwithin{equation}{section}

\begin{document}
\title[Generalized nil-Coxeter algebras over complex reflection
groups]{Generalized nil-Coxeter algebras\\ over discrete complex
reflection groups}

\author{Apoorva Khare}

\address{Department of Mathematics, Indian Institute of Science,
Bangalore -- 560012, India}

\email{khare@iisc.ac.in}

\date{}

\subjclass[2010]{20F55 (Primary), 20F05, 20C08 (Secondary)}

\keywords{Complex reflection group, generalized Coxeter group,
generalized nil-Coxeter algebra, length function}

\begin{abstract}
We define and study generalized nil-Coxeter algebras associated to
Coxeter groups. Motivated by a question of Coxeter (1957), we construct
the first examples of such finite-dimensional algebras that are not the
`usual' nil-Coxeter algebras: a novel $2$-parameter type $A$ family that
we call $NC_A(n,d)$. We explore several combinatorial properties of
$NC_A(n,d)$, including its Coxeter word basis, length function, and
Hilbert--Poincar\'e series, and show that the corresponding generalized
Coxeter group is not a flat deformation of $NC_A(n,d)$. 
These algebras yield symmetric semigroup module categories that are
necessarily not monoidal; we write down their Tannaka--Krein duality.

Further motivated by the Brou\'e--Malle--Rouquier (BMR) freeness
conjecture [\textit{J.~reine angew. math.} 1998], we define generalized
nil-Coxeter algebras $NC_W$ over all discrete real or complex reflection
groups $W$, finite or infinite.
We provide a complete classification of all such algebras that are
finite-dimensional. Remarkably, these turn out to be either the usual
nil-Coxeter algebras, or the algebras $NC_A(n,d)$. This proves as a
special case -- and strengthens -- the lack of equidimensional
nil-Coxeter analogues for finite complex reflection groups.
In particular, generic Hecke algebras are not flat deformations of $NC_W$
for $W$ complex. 
\end{abstract}
\maketitle
\vspace*{-3mm}
\setcounter{page}{2971}

\settocdepth{section}
\tableofcontents

\section{Introduction and main results}

\noindent \textit{Throughout this paper, $\bk$ will denote a fixed unital
commutative ground ring.}\medskip

In this paper we define and study generalized nil-Coxeter algebras
associated to Coxeter groups, and more generally to all discrete complex
reflection groups, finite or infinite. These are algebras that map onto
the associated graded algebras of (generic) Hecke algebras over complex
reflection groups and of Iwahori--Hecke algebras over Coxeter groups.
As we discuss, working with these algebras allows for a broader class
than the corresponding reflection groups.

We begin with real groups. Coxeter groups and their associated Hecke
algebras play an important role in representation theory, combinatorics,
and mathematical physics. Each such group is defined by a Coxeter matrix,
i.e.~a symmetric ``integer'' matrix $M := (m_{ij})_{i,j \in I}$ with $I$
finite and $m_{ii} = 2 \leqslant m_{ij} \leqslant \infty\ \forall i \neq
j$.
The \textit{Artin monoid} $\B_M^{\geqslant 0}$ associated to $M$ is
generated by $\{ T_i : i \in I \}$ modulo the braid relations $T_i T_j
T_i \cdots = T_j T_i T_j \cdots$ for all $i \neq j$ with $m_{ij} <
\infty$, with precisely $m_{ij}$ factors on either side. The Artin group
(or generalized braid group) $\B_M$ is the group generated by these
relations; typically we use $\{ s_i : i \in I \}$ to denote its
generators.
There are three well-studied algebras associated to the matrix $M$: the
group algebra $\bk W(M)$ of the Coxeter group, the $0$-Hecke algebra
\cite{Ca,Nor}, and the nil-Coxeter algebra $NC(M)$ \cite{FS} (also called
the nilCoxeter algebra, nil Coxeter algebra, and nil Hecke ring in the
literature).
These are all free $\bk$-modules, with a ``Coxeter word basis'' $\{ T_w :
w \in W(M) \}$ and length function $\ell(T_w) := \ell(w)$ in $W(M)$; in
each of them the $T_i$ satisfy a quadratic relation.

In a sense, the usual nil-Coxeter algebras $NC(M)$ are better-behaved
than all other generic Hecke algebras (in which $T_i^2 = a_i T_i + b_i$
for scalars $a_i, b_i$, see \cite[Chapter 7]{Hum}): the words $T_w$ have
unique lengths and form a monoid together with $0$. Said differently, the
algebras $NC(M)$ are the only generic Hecke algebras that are graded with
$T_i$ homogeneous of positive degree. Indeed, if $\deg T_i = 1\ \forall
i$, then $NC(M)$ has Hilbert--Poincar\'e polynomial $\prod_{i \in I}
[d_i]_q$, where $[d]_q := \frac{q^d-1}{q-1}$ and $d_i$ are the exponents
of $W(M)$.

We now introduce the main objects of interest in the present work:
generalized Coxeter matrices and their associated nil-Coxeter algebras
(which are always $\mathbb{Z}^{\geqslant 0}$-graded).

\begin{definition}\label{Dnilcox}
Define a \textit{generalized Coxeter matrix} to be a symmetric
`integer' matrix $M := (m_{ij})_{i,j \in I}$ with $I$ finite, $2
\leqslant m_{ij} \leqslant \infty\ \forall i \neq j$, and $m_{ii} <
\infty\ \forall i$. Now fix such a matrix $M$.
\begin{enumerate}
\item Given an integer vector ${\bf d} = (d_i)_{i \in I}$ with $d_i
\geqslant 2\ \forall i$, define $M({\bf d})$ to be the matrix replacing
the diagonal in $M$ by ${\bf d}$.

\item The \textit{generalized Coxeter group} $W(M)$ is the quotient of
the braid group $\B_{M_2}$ by the \textit{order relations} $s_i^{m_{ii}}
= 1\ \forall i$, where $M_2$ is the Coxeter matrix $M((2,\dots,2))$. We
remark that we used the familiar Artin group $\B_{M_2}$ associated to
the Coxeter group $W(M_2)$; we could just as well have written $\B_M$,
since the diagonals of the matrices $M_2$ and $M$ play no role in
$\B_{M_2} = \B_M$.

\item The \textit{braid diagram}, or \textit{Coxeter graph} of $M$ (or of
$W(M)$) has vertices indexed by $I$, and for each pair $i \neq j$ of
vertices, $m_{ij} - 2$ edges between them.

\item Define the corresponding \textit{generalized nil-Coxeter algebra}
as follows (with $M_2$ as above):
\begin{equation}\label{Egen}
NC(M) := \frac{\bk \tangle{T_i, i \in I}}
{(\underbrace{T_i T_j T_i \cdots}_{m_{ij}\ times} =
\underbrace{T_j T_i T_j \cdots}_{m_{ij}\ times}, \
T_i^{m_{ii}} = 0, \ \forall i \neq j \in I)}
= \frac{\bk \B_{M_2}^{\geqslant 0}}{(T_i^{m_{ii}} = 0\ \forall i)},
\end{equation}
where we omit the braid relation $T_i T_j T_i \cdots = T_j T_i T_j
\cdots$ if $m_{ij} = \infty$.

\item Given ${\bf d} = (d_i)_{i \in I}$ as above, define $W_M({\bf d}) :=
W(M({\bf d}))$ and $NC_M({\bf d}) := NC(M({\bf d}))$.
\end{enumerate}
\end{definition}

We are interested in the family of (generalized) nil-Coxeter algebras for
multiple reasons: category theory, real reflection groups, complex
reflection groups, and deformation theory. We elaborate on these
motivations in this section and the next.

\subsection{Tannaka--Krein duality for semigroup categories}\label{S11}

In \cite{Kho}, the representation categories $\Rep NC(A_n)$ were used to
categorify the Weyl algebra. For now we highlight two properties of
generalized nil-Coxeter algebras $NC(M)$ which also have categorical
content:
(i)~for \textit{no} choice of coproduct on $NC(M)$ can it be a bialgebra
(shown below);
and (ii)~every algebra $NC(M)$ is equipped with a cocommutative coproduct
$\Delta : T_i \mapsto T_i \otimes T_i$ for all $i \in I$.

Viewed through the prism of representation categories, the coproduct in
(ii) equips $\Rep NC(M)$ with the structure of a \textit{symmetric}
semigroup category \cite[\S 13,14]{ES}. Note by (i) that the simple
module $\bk$ does not serve as a unit object, whence $\Rep NC(M)$ is
necessarily not monoidal. It is natural to apply the Tannakian formalism
to such categories with ``tensor'' structure. We record the answer which,
while not surprising once formulated, we were unable to find in the
literature.

\begin{definition}\label{Dst}
A \textit{semigroup-tensor category} is a semigroup category
$(\mathcal{C}, \otimes)$ which is also additive and such that $\otimes$
is bi-additive.
\end{definition}

\begin{utheorem}\label{Thm}
Let $A$ be an associative unital algebra over a field $\bk$, $\mathcal{C}
:= \Rep A$, and $F : \mathcal{C} \to {\rm Vec}_\bk$ the forgetful
functor.
\begin{enumerate}
\item Any semigroup-tensor structure on $\mathcal{C}$ together with a
tensor structure on $F$ equips $A$ with a coproduct $\Delta : A \to A
\otimes A$ that is an algebra map.
\item If the semigroup-tensor structure on $\mathcal{C}$ is braided
(respectively, symmetric), then $(A,\Delta)$ is a quasi-triangular
(respectively, triangular) algebra with coproduct. This simply means
there exists an invertible element $R \in A \otimes A$ satisfying the
`hexagon relations'
\[
(1 \otimes \Delta) R = R_{13} R_{12}, \qquad
(\Delta \otimes 1) R = R_{13} R_{23},
\]
and such that $\Delta^{op} = R \Delta R^{-1}$.
Triangularity means further that $R R_{21} = 1 \otimes 1$.
\end{enumerate}
\end{utheorem}

Notice that generalized nil-Coxeter algebras are indeed examples of such
triangular algebras, with a (cocommutative) coproduct but no counit.
Such algebras are interesting in the theory of PBW deformations of smash
product algebras; see Section \ref{Sdeform}. We also show below how to
obtain an ``honest'' symmetric tensor category from each algebra $NC(M)$,
via a central extension.

As noted above, Theorem \ref{Thm} is in a sense ``expected'', and serves
to act more as motivation. That the algebras $NC(M)$ provide concrete
examples of symmetric, nonmonoidal semigroup-tensor categories is novel.
The main results below now focus on the algebras $NC(M)$ themselves.

\subsection{Real reflection groups and novel family of finite-dimensional
nil-Coxeter algebras}

Our next result constructs a novel family of generalized nil-Coxeter
algebras of type $A$, which are finite-dimensional. Classifying the
finite-dimensional objects in Coxeter-type settings in algebra and
combinatorics has been a subject of tremendous classical and modern
interest, including Weyl, Coxeter, and complex reflection groups, their
nil-Coxeter and associated Hecke algebras; but also finite type quivers,
Kleinian singularities, the McKay--Slodowy correspondence, simple Lie
algebras, etc. A very recent setting involves the classification of
finite-dimensional Nichols algebras. Some of the prominent ingredients in
the study of these algebras are common to the present work. See
\cite{GHV,HV1,HV2} for more details.
Another famous recent classification is that of finite-dimensional
pointed Hopf algebras \cite{AnSc}, which turn out to arise from
generalized small quantum groups. With these motivations, our goal is to
similarly classify all generalized nil-Coxeter algebras, and our next
result presents the first novel family of such examples.

We remark that in equation \eqref{Egen}, in generalizing the ``order
relations'' from $T_i^2 = 0$ to $T_i^{m_{ii}} = 0$ we were also motivated
by another such setting: the classical work of Coxeter \cite{Co}, which
investigated generalized Coxeter matrices $M$ for which the group $W(M)$
is finite. 
Specifically, Coxeter considered the (type $A$) Artin braid group
$\B_{A_{n-1}}$, and instead of quotienting by the relations $s_i^2 = 1$
to obtain the symmetric group $S_n$, he worked with $s_i^p = 1\ \forall
i \in I$. Coxeter was interested in computing for which $(n,p)$ is the
quotient group $W_{A_{n-1}}((p,\dots,p))$ a finite group, and what is its
order. He showed (see also \cite{As}) that $W_{A_{n-1}}((p, \dots, p))$
is finite if and only if $\frac{1}{n} + \frac{1}{p} > \frac{1}{2}$, in
which case the size of the quotient group is $\left( \frac{1}{n} +
\frac{1}{p} - \frac{1}{2} \right)^{1-n} \cdot n! / n^{n-1}$.
Coxeter's result was extended by Koster in his thesis \cite{Ko} to
completely classify the ``generalized Coxeter groups'' $W(M)$ that are
finite. These turn out to be precisely the finite Coxeter groups and the
Shephard groups.

Parallel to the above classical works, we wish to understand for which
matrices $M$ is the algebra $NC(M)$ finitely generated as a
$\bk$-module. If $W(M)$ is a Coxeter group, then $\dim NC(M) = |W(M)|$.
Few other answers are known. For instance, Marin \cite{Ma} has shown that
the algebra $NC_{A_2}((m,n))$ is not finitely generated when $m,n
\geqslant 3$. However, apart from the usual nil-Coxeter algebras, to our
knowledge no other finitely generated algebras $NC(M)$ were known to
date.

In the following result, following Coxeter's construction in type $A$
above, we exhibit the first such finite-dimensional family of algebras
$NC(M)$.

\begin{utheorem}\label{ThmA}
Given integers $n \geqslant 1$ and $d \geqslant 2$, define the
$\bk$-algebra
\begin{equation}
NC_A(n,d) := NC_{A_n}((2,\dots,2,d)).
\end{equation}
In other words,
$NC_A(n,d)$ is generated by $T_1, \dots, T_n$, with relations:
\begin{alignat}{5}
T_i T_{i+1} T_i & = T_{i+1} T_i T_{i+1}, & & \forall\ 0 < i < n;\\
T_i T_j & = T_j T_i, & & \forall\ |i-j| > 1;\\
T_1^2 & = \cdots = T_{n-1}^2 & = T_n^d = &\ 0.
\end{alignat}

\noindent Then $NC_A(n,d)$ is a free $\bk$-module, with $\bk$-basis of
$n! (1 + n(d-1))$ generators
\[
\{ T_w : \ w \in S_n \} \ \sqcup \
\{ T_w T_n^k T_{n-1} T_{n-2} \cdots T_{m+1} T_m :\ w \in S_n,\ k \in
[1,d-1],\ m \in [1,n] \}.
\]
In particular, for all $l \in [1,n-1]$, the subalgebra $R_l$ generated by
$T_1, \dots, T_l$ is isomorphic to the usual nil-Coxeter algebra
$NC_{A_l}((2,\dots,2))$.
\end{utheorem}

\begin{remark}
We adopt the following notation in the sequel without further reference:
let
\begin{equation}
w_\circ \in S_{n+1}, \quad w'_\circ \in S_n, \quad w''_\circ \in S_{n-1}
\quad \text{denote the respective longest elements},
\end{equation}

\noindent where the symmetric group $S_{l+1}$ corresponds to the
$\bk$-basis of the algebra $R_l$ for $l = n-2, n-1, n$.
\end{remark}

The algebras $NC_A(n,d)$ have not been studied previously for $d>2$, and
we begin to explore their properties. When $d=2$, $NC_A(n,d)$ specializes
to the usual nil-Coxeter algebra of type $A_n$. In this vein, we present
three properties of $NC_A(n,d)$ akin to the usual nil-Coxeter algebras.

\begin{utheorem}\label{ThmB}
Fix integers $n \geqslant 1$ and $d \geqslant 2$.
\begin{enumerate}
\item The algebra $NC_A(n,d)$ has a length function that restricts to the
usual length function $\ell_{A_{n-1}}$ on $R_{n-1} \simeq
NC_{A_{n-1}}((2,\dots,2))$ (from Theorem \ref{ThmA}), and
\begin{equation}\label{Elength}
\ell(T_w T_n^k T_{n-1} \cdots T_m) = \ell_{A_{n-1}}(w) + k + n-m,
\end{equation}
for all $w \in S_n$, $k \in [1,d-1]$, and $m \in [1,n]$.

\item There is a unique longest word $T_{w'_\circ} T_n^{d-1} T_{n-1}
\cdots T_1$ of length
\[
l_{n,d} := \ell_{A_{n-1}}(w'_\circ) + d+n-2.
\]

\item If $\bk$ is a field, then $NC_A(n,d)$ is local, with unique maximal
ideal $\m$ generated by $T_1, \dots, T_n$. For all $\bk$, the ideal $\m$
is nilpotent with $\m^{1 + l_{n,d}} = 0$.
\end{enumerate}
\end{utheorem}

We also study the algebra $NC_A(n,d)$ in connection to Khovanov's
categorification of the Weyl algebra. See Proposition \ref{Pkhovanov}
below.

\subsection{Complex reflection groups and BMR freeness conjecture}

Determining the finite-dimensionality of the algebras $NC(M)$ is strongly
motivated by the study of complex reflection groups and their Hecke
algebras.
Recall that such groups were enumerated by Shephard--Todd \cite{ST}; 
see also \cite{Coh,LT}.
Subsequently, Popov \cite{Po1} classified the infinite discrete groups
generated by affine unitary reflections; in the sequel we will term these
\textit{infinite complex reflection groups}. For more on these groups,
see e.g.~\cite{BS,Hug1,Hug2,Mal,ORS,ReS} and the references therein.

For complex reflection groups, an important program is the study of
\textit{generic Hecke algebras} over them, as well as the associated
\textit{BMR freeness conjecture} of Brou\'e, Malle, and Rouquier
\cite{BMR1,BMR2} (see also the recent publications \cite{Et,Lo,Ma,MP} and
the thesis \cite{Ch}). This conjecture connects the dimension of a
generic Hecke algebra to the order of the underlying reflection group.
Here we will study this connection for the corresponding nil-Coxeter
algebras, which we define as follows given \cite{Be,BMR2}.

\begin{definition}\label{Dfinite}
Suppose $W$ is a discrete (finite or infinite) complex reflection group,
together with a finite generating set of complex reflections $\{ s_i : i
\in I \}$, the order relations $s_i^{m_{ii}} = 1\ \forall i$, a set of
braid relations $\{ R_j : j \in J \}$ -- each involving words with at
least two distinct reflections $s_i$ -- and for the infinite non-Coxeter
complex reflection groups $W$ listed in \cite[Tables I, II]{Mal}, one
more order relation $R_0^{m_0} = 1$.
Now define $I_0 := I \sqcup \{ 0 \}$ for these infinite non-Coxeter
complex reflection groups $W$, and $I_0 := I$ otherwise. Given an integer
vector ${\bf d} \in \mathbb{N}^{I_0}$ with $d_i \geqslant 2\ \forall i$,
define the corresponding \textit{generalized nil-Coxeter algebra} to be
\begin{equation}
NC_W({\bf d}) := 
\frac{\bk \tangle{T_i, i \in I}}
{( \{ R_j', j \in J \}, T_i^{d_i} = 0\ \forall i \in  I, \
(R'_0)^{d_0} = 0)},
\end{equation}

\noindent where the braid relations $R_j$ are replaced by the
corresponding relations $R'_j$ in the alphabet $\{ T_i : i \in I \}$,
and similarly for $R'_0$ if $R_0^{m_0} = 1$ in $W$. There is also the
notion of the corresponding \textit{braid diagram} as in \cite[Tables
1--4]{BMR2} and \cite[Tables I, II]{Mal}; this is no longer always a
Coxeter graph.
\end{definition}

Note by \cite[\S 1.6]{Po1} that in the above definition, one has to work
with a specific presentation for complex reflection groups, as there is
no canonical (minimal) set of generating reflections. See \cite{Bas1} for
related work.\medskip

There is no known finite-dimensional generalized nil-Coxeter algebra
associated to a finite complex reflection group. Indeed, Marin mentions
in \cite{Ma} that a key difference between real and complex reflection
groups $W$ is the lack of nil-Coxeter algebras of dimension precisely
$|W|$ for the latter. This was verified in some cases for complex
reflection groups in \textit{loc.~cit.}
Our final result shows this assertion -- and in fact a stronger statement
-- for all discrete finite and infinite, real and complex reflection
groups.
Even stronger (\textit{a priori}): we provide a complete classification
of finite-dimensional generalized nil-Coxeter algebras for all such
groups. Notice by \cite[Theorem 1.4]{Po1} that it suffices to consider
only the groups whose braid diagram is connected.

\begin{utheorem}\label{ThmC}
Suppose $W$ is any irreducible discrete real or complex reflection group.
In other words, $W$ is a real reflection group with connected braid
diagram, or a complex reflection group with connected braid diagram and
presentation given as in
\cite[Tables 1--4]{BMR2},
\cite[Tables I, II]{Mal}, or
\cite[Table 2]{Po1}.
Also fix an integer vector ${\bf d}$ with $d_i \geqslant 2\ \forall i$
(including possibly for the additional order relation as in \cite{Mal}).
Then the following are equivalent:
\begin{enumerate}
\item The generalized nil-Coxeter algebra $NC_W({\bf d})$ is finitely
generated as a $\bk$-module.

\item Either $W$ is a finite Coxeter group and $d_i = 2 \ \forall i$, or
$W$ is of type $A_n$ and ${\bf d} = (2, \dots, 2, d)$ or $(d, 2, \dots,
2)$ for some $d>2$.

\item The ideal $\m$ generated by $\{ T_i : i \in I \}$ is nilpotent.
\end{enumerate}

\noindent If these assertions hold, there exists a length function and a
unique longest element in $NC_W({\bf d})$, say of length $l$; now
$\m^{1+l} = 0$.
\end{utheorem}

\noindent In other words, the only finite-dimensional examples (when
$\bk$ is a field) are the usual nil-Coxeter algebras, and the algebras
$NC_A(n,d)$. Note also that all of the above results are
characteristic-free.

A key tool in proving both Theorems \ref{ThmA} and \ref{ThmC} is a
diagrammatic calculus, which is akin to crystal theory from combinatorics
and quantum groups.

\subsection{Further questions and Organization of the paper}

To our knowledge, the algebras $NC_A(n,d)$ for $d>2$ are a novel
construction -- and in light of Theorem \ref{ThmC}, the \textit{only}
finite-dimensional generalized nil-Coxeter algebras other than the
`usual' ones. In particular, a further exploration of their properties is
warranted. We conclude this section by discussing some further
directions.
\begin{enumerate}
\item Nil-Coxeter algebras are related to flag varieties \cite{BGG,KK},
categorification \cite{Kho,KL}, and symmetric function theory \cite{BSS}.
Also recall, the divided difference operator representation of the usual
type $A$ nil-Coxeter algebra $NC_A(n,2)$ is used to define Schubert
polynomials \cite{FS,LS}, and the polynomials the $T_i$ simultaneously
annihilate are precisely the symmetric polynomials. It will be
interesting to determine if $NC_A(n,d), \ d > 2$ has a similar
``natural'' representation as operators on a polynomial ring; and if so,
to consider the polynomials one obtains analogously. (See \cite{Ma} for a
related calculation.) We observe here that for $d>2$, the algebra
$NC_A(n,d)$ does not ``come from'' a finite reflection group, as it is of
larger dimension than the corresponding generalized Coxeter group, by
equation \eqref{Edim} below.

\item Given both the connection to Coxeter groups as well as the crystal
methods used below, it will be interesting to explore if the algebras
$NC_A(n,d)$ are connected to crystals over some Lie (super)algebra.

\item Our proof of Theorem \ref{ThmC} below involves a case-by-case
argument, running over all discrete complex reflection groups. A
type-free proof of this result would be desirable.
%
\end{enumerate}

The paper is organized as follows. In Section \ref{Sdeform} we elaborate
on our motivations and make additional remarks. In the following four
sections we prove, in turn, the four main theorems above.

\section{Background and motivation}\label{Sdeform}



In this section we elaborate on some of the aforementioned motivations
for studying generalized nil-Coxeter algebras and their
finite-dimensionality. First, these algebras are interesting from a
categorical perspective, as their module categories are symmetric
semigroup-tensor categories (see Definition \ref{Dst}) but are not
monoidal. We will discuss in the next section a Tannaka--Krein duality
for such categories, as well as a central extension to a symmetric tensor
category.\medskip

The second motivation comes from real reflection groups: we provide a
novel family of finite-dimensional algebras $NC_A(n,d)$ of type $A$
(akin to the work of Coxeter \cite{Co} and Koster \cite{Ko}).
In this context, it is remarkable (by Theorem \ref{ThmC}) that the
algebras $NC_A(n,d)$ and the usual nil-Coxeter algebras
$NC_W((2,\dots,2))$ are the only finite-dimensional examples.

As Theorem \ref{ThmB} shows, the algebras $NC_A(n,d)$ for $d>2$ are
similar to their ``usual'' nil-Coxeter analogues for $d=2$. Note however
that these algebras also differ in key aspects. See Theorem \ref{Pfrob}
and Proposition \ref{Pprim}, which show in particular that for
$NC_A(n,d)$ with $d > 2$, there are multiple ``maximal'' words, i.e.,
words killed by left- and right-multiplication by every generator $T_i$.
A more fundamental difference arises out of considerations of
\textit{flat deformations}, which we make precise in the remarks around
equation \eqref{Edim} below.\medskip

Our third motivation comes from complex reflection groups and is of much
recent interest: the BMR freeness conjecture, which discusses the
equality of dimensions of generic Hecke algebras and (the group algebra
of) the underlying finite complex reflection group. In this paper we
study the associated graded algebra, i.e.~where all deformation
parameters are set to zero. As shown by Marin \cite{Ma} in some of the
cases, non-Coxeter reflection groups do not come equipped with
finite-dimensional nil-Coxeter analogues. We make this precise in a
strong way in Theorem \ref{ThmC} above, for all complex reflection groups
$W$. In particular, Theorem \ref{ThmC} shows that generic Hecke algebras
are not flat deformations of their underlying (associated graded)
nil-Coxeter analogues for complex $W$. This is a property shared by the
algebras $NC_A(n,d)$ for $d>2$ (but not by Iwahori--Hecke algebras of
Coxeter groups $W = W(M)$, which are flat deformations of $NC(M)$).
Indeed, if $M_{n,d}$ denotes the generalized Coxeter matrix corresponding
to $NC_A(n,d)$, then we claim that:
\begin{equation}\label{Edim}
\dim NC_A(n,d) = n! (1 + n(d-1)) \ > \ | W(M_{n,d}) | =
\begin{cases}
(n+1)!, \quad & \text{if } d>2 \text{ is even,}\\
1, \quad & \text{if } d>2 \text{ is odd.}
\end{cases}
\end{equation}

\noindent To see \eqref{Edim}, if $m_{ij}$ is odd for any generalized
Coxeter matrix $M = M({\bf d})$, then $s_i, s_j$ are conjugate in $W(M)$,
whence $s_i^g = s_j^g = 1$ in $W(M)$ for $g = gcd(d_i,d_j)$. On the other
hand, $NC_M({\bf d})$ surjects onto the nil-Coxeter algebra $NC_M((2,
\dots, 2))$ if $d_i \geqslant 2\ \forall i$. Now if $d_i, d_j \geqslant
2$ are coprime, say, then $s_i$ generates the trivial subgroup of $W(M)$,
while $T_i$ does not vanish in $NC(M)$.\medskip

The generic Hecke algebras discussed above fit in a broader framework of
\textit{deformation theory}, which provides a fourth motivation behind
this paper (in addition to the question of flatness discussed above). The
theory of flat/PBW deformations of associative algebras is an area of
sustained activity, and subsumes Drinfeld Hecke/orbifold algebras
\cite{Dr}, graded affine Hecke algebras \cite{Lu}, 
symplectic reflection algebras and rational Cherednik algebras \cite{EG},
infinitesimal and other Hecke algebras, and other programs in the
literature. We also highlight the program of Shepler and Witherspoon; see
\cite{SW2,SW4} and the references therein. In all of these settings, a
bialgebra $A$ (usually a Hopf algebra) acts on a vector space $V$ and
hence on a quotient $S_V$ of its tensor algebra, and one characterizes
the deformations of this smash-product algebra $A \ltimes S_V$ which are
flat, also termed the ``PBW deformations''.

In this regard, the significance of the generalized nil-Coxeter algebras
$NC(M)$ is manifold. First, the above bialgebra settings were extended in
recent work \cite{Kh} to the framework of ``cocommutative algebras'' $A$,
which also include the algebras $NC_W({\bf d})$. Moreover, we
characterized the PBW deformations of $A \ltimes \Sym(V)$, thereby
extending in \textit{loc.~cit.}~the PBW theorems in the previously
mentioned works. The significance of our framework incorporating $A =
NC_W({\bf d})$ along with the previously studied algebras, is that the
full Hopf/bialgebra structure of $A$ -- specifically, the antipode or
even counit -- is \textit{not} required in order to characterize the flat
deformations of $A \ltimes \Sym(V)$.

Coming to finite-dimensionality, it was shown in the program of Shepler
and Witherspoon (see e.g.~\cite{SW4}), and then in \cite{Kh}, that when
the algebra $A$ with coproduct is finite-dimensional over a field $\bk$,
it is possible to characterize the graded $\bk[t]$-deformations of $A
\ltimes \Sym(V)$, whose fiber at $t=1$ has the PBW property. For $A =
NC_W({\bf d})$, this deformation-theoretic consideration directly
motivates our classification result in Theorem \ref{ThmC}.

We conclude with a third connection to the aforementioned active program
on PBW deformations. We studied in \cite{Kh} the case when
$(A,\m,\Delta)$ is local with $\Delta(\m) \subset \m \otimes \m$. In this
setting, if $\m$ is a \textit{nilpotent} two-sided ideal, then one
obtains a lot of information about the deformations of $A \ltimes
\Sym(V)$, including understanding the PBW deformations, as well as their
center, abelianization, and modules, especially the simple modules. Now
if $A = NC_W({\bf d})$ then $\m$ is generated by the $T_i$; this explains
the interest above in understanding when $\m$ is nilpotent. Theorem
\ref{ThmC} shows that this condition is in fact equivalent to the
generalized nil-Coxeter algebra being finite-dimensional.

\section{Proof of Theorem \ref{Thm}: Tannakian formalism for semigroup
categories}

The remainder of this paper is devoted to proving the four main theorems
in the opening section. We begin by studying the representation category
of $NC(M)$ for a generalized Coxeter matrix $M$. The first assertion is
that this category can never be a monoidal category in characteristic
zero, and it follows from the following result.

\begin{prop}\label{Pno-counit}
Suppose $\bk$ is a field of characteristic zero and $M$ is a generalized
Coxeter matrix. Then $NC(M)$ is not a bialgebra.
\end{prop}

The result fails to hold in positive characteristic. Indeed, for any
prime $p \geqslant 2$ the algebra $(\mathbb{Z} / p \mathbb{Z})[T] /
(T^p)$ is a bialgebra, with coproduct $\Delta(T) := 1 \otimes T + T
\otimes 1$ and counit $\varepsilon(T) := 0$.

\begin{proof}
Note there is a unique possible counit, $\varepsilon : T_i \mapsto 0\
\forall i \in I$. Now suppose $\Delta : NC(M) \to NC(M) \otimes NC(M)$ is
such that
\[
(\id \otimes \varepsilon) \circ \Delta = \id = (\varepsilon \otimes \id)
\circ \Delta
\]

\noindent on $NC(M)$. Setting $\m := \ker \varepsilon$ to be the ideal
generated by $\{ T_i : i \in I \}$, it follows that
\begin{equation}\label{Einclusion}
\Delta(T_i) \in 1 \otimes T_i + T_i \otimes 1 + \m \otimes \m.
\end{equation}

\noindent Note that $\m \otimes \m$ constitutes the terms of higher
`total degree' in $\Delta(T_i)$, in the $\mathbb{Z}^{\geqslant
0}$-grading on $NC(M)$. Now if $\Delta$ is multiplicative, then raising
\eqref{Einclusion} to the $m_{ii}$th power yields:
\[
0 = \Delta(T_i)^{m_{ii}} = \sum_{k=1}^{m_{ii}-1} \binom{m_{ii}}{k} T_i^k
\otimes T_i^{m_{ii} - k} + \ \text{higher degree terms}.
\]

\noindent This is impossible as long as the image of $T_i$ in $NC(M)$ is
nonzero; assuming this, it follows $\Delta$ cannot be multiplicative,
hence not a coproduct on $NC(M)$. Finally, $NC(M)$ surjects onto the
usual nil-Coxeter algebra $NC(M_2)$ with $M_2 = M((2,\dots,2))$. As
$NC(M_2)$ has a Coxeter word basis indexed by $W(M)$, it follows that
$T_i$ is indeed nonzero in $NC(M)$.
\end{proof}

As a consequence of Proposition \ref{Pno-counit} and the Tannakian
formalism in \cite[Theorem 18.3]{ES}, for any generalized Coxeter matrix
$M$ the module category $\Rep NC(M)$ is necessarily not a tensor
category. That said, the map $\Delta : T_i \mapsto T_i \otimes T_i$
is a coproduct on $NC(M)$, i.e.~a coassociative algebra map.
The cocommutativity of $\Delta$ implies $\Rep NC(M)$ is a symmetric
semigroup category. We now outline how to show the first theorem above,
which seeks to understand Tannaka--Krein duality for such categories
(possibly without unit objects).

\begin{proof}[Proof of Theorem \ref{Thm}]
The proof of part (1) follows that of \cite[Theorem 18.3]{ES}; one now
ignores the last statement in that proof. The additional data required in
the two braided versions in part (2) can be deduced from the proof of
\cite[Proposition 14.2]{ES}.
\end{proof}

We conclude this section by passing from $\Rep NC(M)$ to an ``honest''
tensor category -- say with $\bk$ a field. Alternately, via the Tannakian
formalism in \cite[Theorem 18.3]{ES}, we produce a bialgebra
$\widetilde{NC}(M)$ that surjects onto $NC(M)$. Namely,
$\widetilde{NC}(M)$ is generated by $\{ T_i : i \in I \}$ and an
additional generator $T_\infty$, subject to the braid relations on the
former set, as well as
\[
T_i^{m_{ii}} = T_i T_\infty = T_\infty T_i = T_\infty^2 := T_\infty, \
\forall i \in I.
\]

\noindent Note that $\widetilde{NC}(M)$ is no longer
$\mathbb{Z}^{\geqslant 0}$-graded; but it is a central extension:
\[
0 \to \bk T_\infty \to \widetilde{NC}(M) \to NC(M) \to 0.
\]

\noindent Now asking for all $T_i$ and $T_\infty$ to be grouplike yields
a unique bialgebra structure on $\widetilde{NC}(M)$:
\[
\widetilde{\Delta} : T_i \mapsto T_i \otimes T_i, \quad T_\infty \mapsto
T_\infty \otimes T_\infty, \qquad \widetilde{\varepsilon} : T_i, T_\infty
\mapsto 1,
\]

\noindent and hence a monoidal category structure on $\Rep
\widetilde{NC}(M)$, as claimed.

\section{Proof of Theorem \ref{ThmA}: Distinguished basis of
words}\label{SA}

We now prove our main theorems on the algebras $NC(M)$ -- specifically,
the family $NC_A(n,d)$ -- beginning with Theorem \ref{ThmA}. Note that if
$d=2$ then the algebra $NC_A(n,d)$ is the usual nil-Coxeter algebra,
while if $n=1$ then the algebra is $\bk[T_1] / (T_1^d)$. Theorems
\ref{ThmA} and \ref{ThmB} are easily verified for these cases, e.g. using
\cite[Chapter 7]{Hum}. Thus, we assume throughout their proofs below that
$n \geqslant 2$ and $d \geqslant 3$.

We begin by showing the $\bk$-rank of $NC_A(n,d)$ is at most $n! (1 +
n(d-1))$. Notice that $NC_A(n,d)$ is spanned by words in the $T_i$. We
now \textbf{claim} that a word in the $T_i$ is either zero in
$NC_A(n,d)$, or equal by the braid relations to a word in which all
occurrences of $T_n$ are successive, in a monomial $T_n^k$ for some $1
\leqslant k \leqslant d-1$.

To show the claim, consider a word $\T := \cdots T_n^a T_w T_n^b \cdots$,
where $a,b > 0$ and $T_w = T_{i_1} \cdots T_{i_k}$ is a word in $T_1,
\dots, T_{n-1}$. Rewrite $\T$ using the braid relations if required, so
that $w \in S_n$ has minimal length, say $k$. We may assume $k>0$, else
we would be done.
Now using the braid relations $T_i T_n = T_n T_i$ for $i \leqslant n-2$,
further assume that $i_1 = i_k = n-1$ (otherwise the factors may be
`moved past' the $T_n$ using the braid relations). Similarly, $i_2 =
i_{k-1} = n-2$, and so on. Thus, if $T_w \neq 0$, then assume by the
minimality of $\ell(w)$ that
\[
T_w = T_{n-1} T_{n-2} \cdots T_{m+1} T_m T_{m+1} \cdots T_{n-2} T_{n-1},
\quad \text{for some } 1 \leqslant m \leqslant n-1.
\]

We next claim that the following relation holds in the Artin braid group
$\B_n$, hence in $NC_{A_n}({\bf d})$ for any ${\bf d}$:
\begin{equation}\label{Einduct}
T_{n-1} \cdots T_m \cdots T_{n-1} = T_m T_{m+1} \cdots T_{n-2} T_{n-1}
T_{n-2} \cdots T_{m+1} T_m.
\end{equation}

\noindent This is shown by descending induction on $m \leqslant n-1$.
Hence,
\begin{align}\label{Ereduce}
&\ T_n^a \cdot (T_{n-1} \cdots T_m \cdots T_{n-1}) \cdot T_n^b\\
= &\ T_n^a \cdot (T_m \cdots T_{n-2} T_{n-1} T_{n-2} \cdots T_m) \cdot
T_n^b\notag\\
= &\ (T_m \cdots T_{n-2}) T_n^{a-1} (T_n T_{n-1} T_n) T_n^{b-1} (T_{n-2}
\cdots T_m)\notag\\
= &\ (T_m \cdots T_{n-2}) T_n^{a-1} (T_{n-1} T_n T_{n-1}) T_n^{b-1}
(T_{n-2} \cdots T_m).\notag
\end{align}

\noindent If $\max(a,b) = 1$ then the claim follows; if $a>1$ then the
last expression contains the substring $(T_n T_{n-1} T_n) T_{n-1} =
T_{n-1} T_n T_{n-1}^2 = 0$; and similarly if $b>1$. This shows the
claim.\smallskip

We now prove the upper bound on the $\bk$-rank. Notice that $T_1, \dots,
T_{n-1}$ generate a subalgebra $R_{n-1} \subset NC_A(n,d)$ in which the
nil-Coxeter relations for $W_{A_{n-1}} = S_n$ are satisfied. Hence the
map $: NC_{A_{n-1}}((2, \dots, 2)) \twoheadrightarrow R_{n-1} :=
\tangle{T_1, \dots, T_{n-1}}$ is an algebra map.

Now notice by equation \eqref{Ereduce} that every nonzero word in
$NC_A(n,d) \setminus R_{n-1}$ is of the form $\T = T_w T_n^k T_{w'}$,
where $1 \leqslant k \leqslant d-1$, $w,w' \in W_{A_{n-1}}$, and hence
$T_w, T_{w'} \in R_{n-1}$. By a similar reasoning as above, assuming $w'$
of minimal length in $S_{n-1}$, we may rewrite $\T$ such that $T_{w'} =
T_{n-1} \cdots T_m$ for some $1 \leqslant m \leqslant n$.
Carrying out this operation yields $T_{w''} T_n^k T_{n-1} \cdots T_m$ for
some reduced word $w'' \in W_{A_{n-1}}$ (i.e., such that $T_{w''}$ is
nonzero in $NC_{A_{n-1}}((2, \dots, 2))$). Thus,
\[
NC_A(n,d) = R_{n-1} + \sum_{k=1}^{d-1} \sum_{m=1}^n R_{n-1} \cdot T_n^k
\cdot (T_{n-1} \cdots T_m).
\]

\noindent As $R_{n-1}$ has at most $n!$ generators, it follows that
$NC_A(n,d)$ has at most $(1 + n(d-1)) \cdot n!$ generators, which shows
the desired upper bound on its $\bk$-rank.\medskip

The hard part of the proof involves showing that the words $T_w T_n^k
T_{n-1} \cdots T_m$ form a $\bk$-basis of $NC_A(n,d)$.
We will require the following technical lemma on the symmetric group and
its nil-Coxeter algebra. A proof is included for completeness.

\begin{lemma}\label{Lsymm}
Suppose $W = W_{A_{n-1}} = S_n$ is the symmetric group, with simple
reflections $s_1, \dots, s_{n-1}$ labelled as usual. Then every element
$w$ of $W_{A_{n-1}} \setminus W_{A_{n-2}} = S_n \setminus S_{n-1}$ can be
written in reduced form as $w = w' s_{n-1} \cdots s_{m'}$, where $w' \in
S_{n-1} = W_{A_{n-2}}$ and $m' \in [1,n-1]$ are unique. Given such an
element $w \in S_n$, we have in the usual nil-Coxeter algebra
$NC_{A_n}((2,\dots,2))$:
\begin{equation}\label{EnilcoxA}
T_n \cdot T_w \cdot T_n \cdots T_m = \begin{cases}
T_{w'} T_{n-1} \cdots T_{m-1} \cdot T_n \cdots T_{m'}, & \qquad \text{if
} m' < m,\\
0 & \qquad \text{otherwise}.
\end{cases}
\end{equation}
\end{lemma}

\noindent Note that equation \eqref{EnilcoxA} can be thought of as a
statement on lengths in the symmetric group.

\begin{proof}
We first claim that $w \in W_{A_{n-1}} \setminus W_{A_{n-2}}$ has a
reduced expression in which $s_{n-1}$ occurs exactly once. The proof is
by induction on $n$: clearly the claim is true for $n=2$. Now given the
claim for $n-2 \geqslant 2$, consider any reduced expression for $w$ that
contains a sub-word $s_{n-1} w'' s_{n-1}$, where $w \in W_{A_{n-2}}$. By
the induction hypothesis, $w'' = w' s_{n-2} \cdots s_m$ for some $w' \in
W_{A_{n-3}}$ and $m \in [1,n-1]$. Hence if $m \leqslant n-2$, then
\[
w = \cdots s_{n-1} \left( w' s_{n-2} s_{n-3} \cdots s_m \right) s_{n-1}
= \cdots w' (s_{n-1} s_{n-2} s_{n-1}) (s_{n-3} \cdots s_m) \cdots,
\]

\noindent and by the braid relations, this equals a reduced expression
for $w \in W_{A_{n-1}}$, with one less occurrence of $s_{n-1}$. A similar
analysis works if $m = n-1$.
Repeatedly carrying out this procedure proves the claim.

We can now prove the uniqueness of $w',m'$ as in the lemma. By the
previous paragraph, write $w \in W_{A_{n-1}} \setminus W_{A_{n-2}}$ as $w
= w_1 s_{n-1} w_2$, with $w_1, w_2 \in W_{A_{n-2}}$ and $w_2$ of smallest
possible length, say $w_2 = s_{i_1} \cdots s_{i_k}$ for $i_1, \dots, i_k
\leqslant n-2$. Using the braid relations, clearly $i_1 = n-2$, hence
$i_2 = n-3$ (by minimality of $k$). Choose the smallest $l \geqslant 3$
such that $i_l \neq n-1-l$. We now produce a contradiction assuming that
such an integer $l$ exists. If $i_l < n-1-l$, then we may move $s_{i_l}$
past the preceding terms, contradicting the minimality of $k$. Clearly
$i_l \neq n-l$, else $w_2$ was not reduced. Thus $i_l > n-1-l$, whence
$w_2$ is of the form $s_{n-2} \cdots s_{i_l} s_{i_l - 1} s_{i_l} \cdots$.
Now verify in $W_{A_{n-1}}$ that
\begin{align*}
w = w_1 s_{n-1} s_{n-2} \cdots s_{i_l + 1} s_{i_l} s_{i_l - 1} s_{i_l}
\cdots = &\ w_1 s_{n-1} \cdots s_{i_l+1} s_{i_l-1} s_{i_l} s_{i_l-1}
\cdots\\
= &\ w_1 s_{i_l-1} \cdot s_{n-1} \cdot s_{n-2} \cdots s_{i_l+1} s_{i_l}
s_{i_l-1} \cdots,
\end{align*}

\noindent which contradicts the minimality of $k$. Thus such an integer
$l$ cannot exist, which proves that $w = w_1 s_{n-1} \cdots s_{m'}$ for
some $m' \in [1,n-1]$.

We next claim that the integer $m'$ is unique for $w \in W_{A_n}
\setminus W_{A_{n-1}}$. We first make the sub-claim that if $w \in
W_{A_{n-1}}$ is reduced, then so is  $w s_n s_{n-1} \cdots s_m$. To see
why, first recall \cite[Lemma 1.6, Corollary 1.7]{Hum}, which together
imply that if $w \alpha > 0$ for any finite Coxeter group $W$, any $w \in
W$, and any simple root $\alpha > 0$, then $\ell(w s_\alpha) = \ell(w) +
1$. Now the sub-claim follows by applying this result successively to $(w
s_n \cdots s_{j+1}, \alpha_j)$ for $j=n, n-1, \cdots, m$.
Next, define $C_m := W_{A_{n-1}} \cdot s_n s_{n-1} \cdots s_m$, with
$C_{n+1} := W_{A_{n-1}}$. It follows by the sub-claim above that $|C_m| =
|W_{A_{n-1}}| = n!$ for all $m$. Hence,
\[
(n+1)! = |W_{A_n}| \leqslant \sum_{m=1}^{n+1} |C_m| \leqslant
\sum_{m=1}^{n+1} n! = (n+1)!.
\]

\noindent This shows that $W_{A_n} = \bigsqcup_{m=1}^{n+1} C_m$, which
proves the uniqueness of $m$ in the above claim. Now write $w_1$ in
reduced form to obtain that $w' = w s_{m'} \cdots s_{n-1}$ is also
unique.

It remains to show equation \eqref{EnilcoxA} in $NC_{A_n}((2, \dots,
2))$. Using the above analysis, write $T_w = T_{w'} T_{n-1} \cdots
T_{m'}$; since $T_n$ commutes with $T_{w'}$, we may assume $w' = 1$.
First suppose $m' \geqslant m$. Then it suffices to prove that $(T_n
\cdots T_{m'})^2 = 0$ for all $1 \leqslant m' \leqslant n$. Without loss
of generality we may work in the subalgebra generated by $T_{m'}, \dots,
T_n$, and hence suppose $m'=1$. We now prove by induction that $(T_n
\cdots T_1)^2 = 0$. This is clear if $n=1,2$, and for $n>2$,
\[
(T_n \cdots T_1)^2 = T_n T_{n-1} T_n \cdot T_{n-2} \cdots T_1 \cdot
T_{n-1} \cdots T_1
= T_{n-1} T_n \cdot (T_{n-1} \cdots T_1)^2 = 0.
\]

Next suppose $m' < m$; once again we may suppose $m'=1$. We prove the
result by induction on $n$, the base case of $n=2$ (and $m=2$) being
easy. Thus, for $1 < m \leqslant n$, we compute:
\begin{align*}
T_n \cdots T_1 \cdot T_n \cdots T_m = &\ T_n T_{n-1} T_n \cdot T_{n-2}
\cdots T_1 \cdot T_{n-1} \cdots T_m\\
= &\ T_{n-1} T_n \cdot \left( T_{n-1} \cdots T_1 \right) \cdot
\left( T_{n-1} \cdots T_m \right)\\
= &\ T_{n-1} T_n \cdot \left( T_{n-2} \cdots T_{m-1} \right) \cdot
\left( T_{n-1} \cdots T_1 \right)\\
= &\ T_{n-1} T_{n-2} \cdots T_{m-1} \cdot T_n T_{n-1} \cdots T_1.
\qedhere
\end{align*}
\end{proof}

\begin{remark}\label{Rsymm}
Notice that equation \eqref{EnilcoxA} holds in any algebra containing
elements $T_1, \dots, T_n$ that satisfy the braid relations and $T_1^2 =
0$. In particular, \eqref{EnilcoxA} holds in $NC_A(n,d)$ for $n>1$.
\end{remark}

Returning to the proof of Theorem \ref{ThmA}, we now introduce a
diagrammatic calculus akin to crystal theory.
We first write out the $n=2$ case, in order to provide intuition for the
case of general $n$.
Let $\scrm$ be a free $\bk$-module, with basis given by the nodes in the
graph in Figure \ref{Fig1} below.

\begin{figure}[ht]
\includegraphics[width=12.5cm]{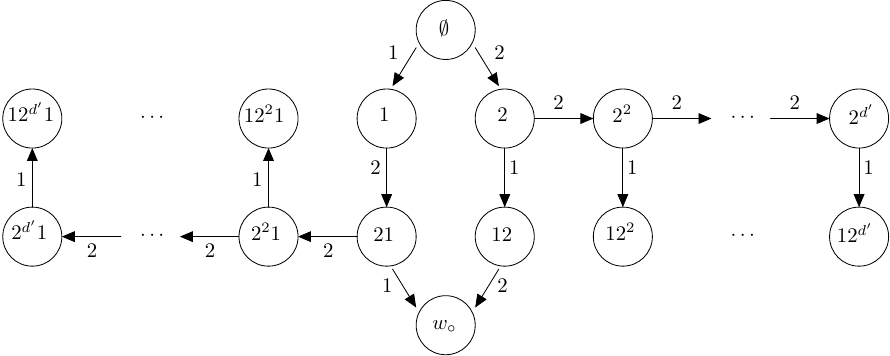}
\caption{Regular representation for $NC_A(2,d)$, with $d' = d-1$}
\label{Fig1}
\end{figure}

In the figure, the node $1 2^2 1$ should be thought of as $T_1 T_2^2 T_1$
(applied to the unit $1_{NC_A(2,d)}$, i.e., to the generating basis
vector corresponding to $\emptyset$), and similarly for the other nodes.
The arrows denote the action of $T_1$ and $T_2$; all remaining generator
actions on nodes yield zero. Now one verifies by inspection that the
defining relations in $NC_A(2,d)$ are satisfied by this action on
$\scrm$. Therefore $\scrm$ is an $NC_A(2,d)$-module of $\bk$-rank $4d-2 =
2!(1 + 2(d-1))$. Since $\scrm$ is generated by the basis vector
corresponding to the node $\emptyset$, we have a surjection $: NC_A(2,d)
\twoheadrightarrow \scrm$ that sends $T_1, T_2^k, T_1 T_2^k, T_2^k T_1,
T_1 T_2^k T_1$ to the corresponding basis vectors in the free
$\bk$-module $\scrm$. Now the result for $n=2$ follows by the upper bound
on the $\bk$-rank, proved above.\smallskip

The strategy is similar for general $n$, but uses the following more
detailed notation. For each $w \in S_l$ with $l \leqslant n$, let $T_w$
denote the corresponding (well-defined) word in the alphabet $\{ T_1,
\dots, T_{l-1} \}$, and let $R_{l-1}$ denote the subalgebra of
$NC_A(n,d)$ generated by these letters. Now define a free $\bk$-module
$\scrm$ of $\bk$-rank $n! (1 + n(d-1))$, with basis elements the set of
words
\begin{equation}\label{Ebasis}
\B := \{ B(w,k,m) : w \in S_n,\ k \in [1, d-1],\ m \in [1,n] \} \sqcup \{
B(w) : w \in S_n \}.
\end{equation}

\noindent We observe here that the basis vectors $B(w,k,m), B(w)$ are to
be thought of as corresponding respectively to the words
\begin{equation}\label{Estandard}
T_w T_n^k T_{n-1} \cdots T_m, \ \ T_w, \qquad w \in S_n, \ 
k \in [1, d-1],\ m \in [1,n].
\end{equation}

\begin{definition}
An expression for a word in $NC_A(n,d)$ of the form \eqref{Estandard}
will be said to be in \textit{standard form}.
\end{definition}


We now define an $NC_A(n,d)$-module structure on $\scrm$, via defining a
directed graph structure on $\B$ (or more precisely, on $\B \sqcup \{ 0
\}$) that we now describe.
The following figure (Figure \ref{Fig2} below) may help in visualizing
the structure. The figure should be thought of as analogous to the
central hexagon and either of the two ``arms'' in Figure \ref{Fig1}.

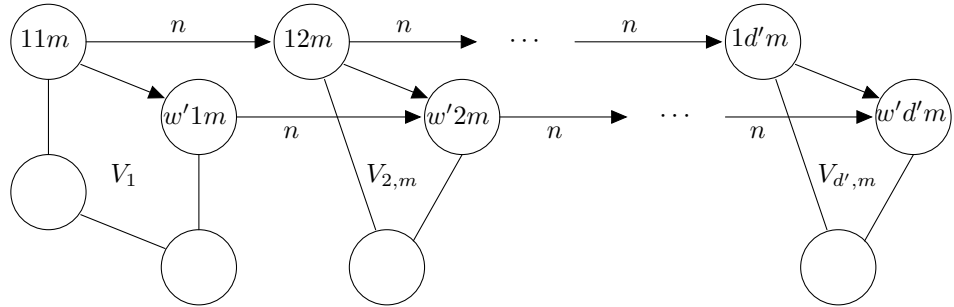
\begin{figure}[ht]
\begin{tikzpicture}[line cap=round,line join=round,>=triangle 45,x=1.0cm,y=1.0cm]
\draw(1.5,1.5) circle (0.5cm);
\draw(1.5,3.5) circle (0.5cm);
\draw(3.5,2.5) circle (0.5cm);
\draw(3.5,0.5) circle (0.5cm);
\draw(5,3.5) circle (0.5cm);
\draw(6,0.5) circle (0.5cm);
\draw(7,2.5) circle (0.5cm);
\draw(11,3.5) circle (0.5cm);
\draw(12,0.5) circle (0.5cm);
\draw(13,2.5) circle (0.5cm);
\draw [->] (1.93,3.2) -- (3.05,2.75);
\draw [->] (5.43,3.2) -- (6.55,2.75);
\draw [->] (11.43,3.2) -- (12.55,2.75);
\draw [->] (2,3.5) -- (4.46,3.5);
\draw [->] (4,2.5) -- (6.46,2.5);
\draw [->] (5.5,3.5) -- (7.2,3.5);
\draw [->] (7.5,2.5) -- (9.2,2.5);
\draw [->] (8.5,3.5) -- (10.46,3.5);
\draw [->] (10.5,2.5) -- (12.46,2.5);
\draw (1.5,3)-- (1.5,2);
\draw (1.93,1.2)-- (3.05,0.75);
\draw (3.5,2)-- (3.5,1);
\draw (5.17,3)-- (5.85,1);
\draw (7,2)-- (6.35,0.85);
\draw (11.17,3)-- (11.85,1);
\draw (13,2)-- (12.35,0.85);
\draw (1,3.78) node[anchor=north west] {$ 11m $};
\draw (2.9,2.8) node[anchor=north west] {$ w'1m $};
\draw (4.5,3.78) node[anchor=north west] {$ 12m $};
\draw (6.4,2.8) node[anchor=north west] {$ w'2m $};
\draw (10.45,3.83) node[anchor=north west] {$ 1d'm $};
\draw (12.37,2.84) node[anchor=north west] {$ w'd'm $};
\draw (7.5,3.71) node[anchor=north west] {$ \cdots $};
\draw (9.5,2.71) node[anchor=north west] {$ \cdots $};
\draw (2.2,2) node[anchor=north west] {$ V_1 $};
\draw (5.6,2) node[anchor=north west] {$ V_{2,m} $};
\draw (11.6,2) node[anchor=north west] {$ V_{d',m} $};
\draw (4.5,2.5) node[anchor=north west] {$n$};
\draw (8,2.5) node[anchor=north west] {$n$};
\draw (10.7,2.5) node[anchor=north west] {$n$};
\draw (3,3.9) node[anchor=north west] {$n$};
\draw (6,3.9) node[anchor=north west] {$n$};
\draw (9,3.9) node[anchor=north west] {$n$};
\end{tikzpicture}
\caption{Regular representation for $NC_A(n,d)$, with $d' = d-1$}
\label{Fig2}
\end{figure}

We begin by explaining the figure. Each node $(wkm)$ (or $(w)$)
corresponds to the basis vector $B(w,k,m)$ (or $B(w)$). Notice that the
vectors $\{ B(w,1,m) \} \sqcup \{ B(w) \}$ are in bijection with the
Coxeter word basis of the usual nil-Coxeter algebra $NC_{A_n}((2, \dots,
2))$. Let $V_1$ denote their span, of $\bk$-rank $(n+1)!$. Now given $1
\leqslant m \leqslant n$ and $1 \leqslant k \leqslant d-1 =: d'$, define
$V_{k,m}$ to be the span of the basis elements $\{ B(w,k,m) : w \in S_n
\}$, of $\bk$-rank $n!$. Then $\scrm = V_1 \oplus \bigoplus_{k>1,m}
V_{k,m}$. Note as a special case that in Figure \ref{Fig1},
the central hexagon spans $V_1$,
the nodes $2^k 1, 1 2^k 1$ span $V_{k,1}$, and
$2^k, 1 2^k$ span $V_{k,2}$.
We now define the $NC_A(n,d)$-action:
\begin{itemize}
\item Let $V_{1,n+1}$ denote the $\bk$-span of $\{ B(w) : w \in S_n \}$.
Then for $1 \leqslant m \leqslant n+1$, each $V_{1,m}$ has a
distinguished basis in bijection with $S_n$; the same holds for each
$V_{k,m}$ with $k \in [2,d-1]$ and $m \in [1,n]$. Now equip all of the
above spaces $V_{k,m}$ with the corresponding module structure over the
usual nil-Coxeter algebra of type $A_{n-1}$. Such a structure is uniquely
determined, if given $w = s_{i_1} \cdots s_{i_l} \in S_n$ with all $i_j <
n$, we set $T_w \cdot B(1,k,m) := B(w,k,m)$ and $T_w \cdot B(1) := B(w)$.

\item We next define the action of $T_n$ on $\scrm$. Via Lemma
\ref{Lsymm}, write $w \in S_n$ as $w' s_{n-1} \cdots s_{m'}$ with $w',m'$
unique. Now using the previous paragraph, it follows that $B(w,k,m) =
T_{w'} T_{n-1} \cdots T_{m'} \cdot B(1,k,m)$. Correspondingly, if $w \in
S_{n-1}$ (i.e., $m' = n$), define
\[
T_n \cdot B(w,k,m) := {\bf 1}(k \leqslant d-2) B(w,k+1,m), \qquad
T_n \cdot B(w) := B(w,1,n).
\]

\item On the other hand, suppose $m' \leqslant n-1$. If $k \geqslant 2$,
then define $T_n \cdot B(w,k,m) := 0$. Otherwise define $T_n \cdot B(w)
:= B(w',1,m')$ with $w',m'$ as in Lemma \ref{Lsymm}, and (see Equation
\ref{EnilcoxA}):
\begin{equation}\label{Ereln}
T_n \cdot B(w,1,m) := \begin{cases}
B(w' s_{n-1} \cdots s_{m-1},1,m'), & \qquad \text{if } m' < m,\\ 0 &
\qquad \text{otherwise}.
\end{cases}
\end{equation}
\end{itemize}

It remains to show that the above graph structure indeed defines an
$NC_A(n,d)$-module structure on $\scrm$; then a similar argument as above
(in the $n=2$ case) completes the proof. In the following argument, we
will occasionally use Lemma \ref{Lsymm} (as well as Remark \ref{Rsymm})
without reference. First notice that the algebra relations involving only
$T_1, \dots, T_{n-1}$ are clearly satisfied on $\scrm$ as it is a
$R_{n-1}$-free module by construction. To verify that the relations
involving $T_n$ hold on $\scrm$, notice (e.g. via Figure \ref{Fig2}) that
the $\bk$-basis $\B$ of $\scrm$ can be partitioned into three subsets:
\begin{align}\label{Esubsets}
\B_1 := &\ \{ B(w,k,m) : k \geqslant 1, m \in [1,n], w \in S_{n-1} \}
\sqcup \{ B(w) : w \in S_{n-1} \},\notag\\
\B_2 := &\ \{ B(w,k,m) : k \geqslant 2, m \in [1,n], w \in S_n \setminus
S_{n-1} \},\\
\B_3 := &\ \{ B(w,1,m) : m \in [1,n], w \in S_n \setminus S_{n-1} \}
\sqcup \{ B(w) : w \in S_n \setminus S_{n-1} \}.\notag
\end{align}

Recall by the opening remarks in Section \ref{SA} that $n \geqslant 2$
and $d \geqslant 3$. We first show that the relation $T_n^d = 0$ holds as
an equality of linear operators on each vector $b \in \B$, and hence on
the $\bk$-module $\scrm$. We separately consider the cases $b \in \B_i$
for $i=1,2,3$, as in \eqref{Esubsets}.
\begin{enumerate}
\item If $b = B(w,k,m) \in \B_1$, then $b$ lies in the ``top rows'' of
Figure \ref{Fig2}. It is easily verified that $T_n^{d-k-1} \cdot B(w,k,m)
= B(w,d-1,m)$, and this is killed by $T_n$, as desired. The same
reasoning shows that $T_n^d$ kills $b = B(w)$ for $w \in S_{n-1}$.

\item Let $b = B(w,k,m) \in \B_2$. Then the relation holds on $b$ since
$T_n \cdot B(w,k,m)$ $= 0$. (These correspond to vectors in $V_{k,m}$ for
$k \geqslant 2$, which do not lie in the ``top rows'' in Figure
\ref{Fig2}.)

\item Finally, let $b \in \B_3$; thus $w \in S_n \setminus S_{n-1}$, and
we write $w = w' s_{n-1} \cdots s_{m'}$ by Lemma \ref{Lsymm}. It follows
from Remark \ref{Rsymm} that $T_n^2 \cdot B(w,1,m) = 0$ and $T_n^d \cdot
B(w) = T_n^{d-1} \cdot B(w',1,m') = 0$.
\end{enumerate}

We next show that the relation $T_i T_n = T_n T_i$ holds on $\B$ for all
$i \leqslant n-2$. We consider the same three cases as in
\eqref{Esubsets}.
\begin{enumerate}
\item Fix $w \in S_{n-1}$. If $b = B(w,k,m)$ with $k \geqslant 1$, then
verify using the aforementioned action that both $T_i T_n \cdot B(w,k,m)$
and $T_n T_i \cdot B(w,k,m)$ equal $B(s_i w, k+1, m)$ if $\ell(s_i w) >
\ell(w)$ and $k \leqslant d-2$, and $0$ otherwise.
Similarly,
\[
T_i T_n \cdot B(w) = {\bf 1}(\ell(s_i w) > \ell(w)) B(s_i w, 1, n) = T_n
T_i \cdot B(w).
\]

\item Let $b = B(w,k,m)$ with $w \in S_n \setminus S_{n-1}$ and $k
\geqslant 2$. Then $T_i T_n \cdot B(w,k,m) = 0$. To compute $T_n T_i
\cdot B(w,k,m)$, since $i \leqslant n-2$, it follows that $s_i w \in S_n
\setminus S_{n-1}$. If $T_i T_w = 0$ then we are done since $B(w,k,m) =
T_w \cdot B(1,k,m)$. Else note that $s_i w \in S_n \setminus S_{n-1}$,
whence $T_n T_i \cdot B(w,k,m) = T_n \cdot B(s_i w, k,m) = 0$ from above.

\item Finally, let $w \in S_n \setminus S_{n-1}$ and write $w = w'
s_{n-1} \cdots s_{m'}$ by Lemma \ref{Lsymm}. First suppose $b =
B(w,1,m)$. If $\ell(s_i w) < \ell(w)$, then it is not hard to show that
both $T_i T_n \cdot B(w,1,m)$ and $T_n T_i \cdot B(w,1,m)$ vanish.
Otherwise both terms are equal to $B(s_i w' s_{n-1} \cdots
s_{m-1},1,m')$. A similar analysis shows that if $\ell(s_i w) < \ell(w)$,
then $T_i T_n \cdot B(w) = T_n T_i \cdot B(w) = 0$, otherwise $T_i T_n
\cdot B(w) = T_n T_i \cdot B(w) = B(s_i w', 1, m')$.
\end{enumerate}

Next, we show that the braid relation $T_{n-1} T_n T_{n-1} = T_n T_{n-1}
T_n$ holds on $\B$. This is the most involved computation to carry out.
We consider the same three cases as above.
\begin{enumerate}
\item Fix $w \in S_{n-1}$. If $b = B(w,k,m)$ with $k \geqslant 2$, then
it is easily verified that both sides of the braid relation kill
$B(w,k,m)$. If instead $k=1$, then
\[
T_n T_{n-1} T_n \cdot B(w,1,m) = T_n T_{n-1} \cdot B(w,2,m) = 0.
\]

\noindent To compute the other side, first notice $B(w,1,m) = T_n \cdot
B(w s_{n-1} \cdots s_m)$. Hence,
\[
T_n T_{n-1} \cdot B(w,1,m) = T_n T_{n-1} T_n \cdot B(w s_{n-1} \cdots
s_m).
\]

\noindent Now if the braid relation holds on $B(w)$ for all $w \in S_n$,
then
\begin{align*}
T_{n-1} T_n T_{n-1} \cdot B(w,1,m) = &\ T_{n-1} \cdot T_n T_{n-1} T_n
\cdot B(w s_{n-1} \cdots s_m)\\
= &\ T_{n-1} T_{n-1} T_n T_{n-1} \cdot B(w s_{n-1} \cdots s_m) \in
T_{n-1}^2 \cdot \scrm = 0,
\end{align*}

\noindent where the last equality follows from the definition of $\scrm$
as an $R_{n-1}$-module. It thus suffices for this case to verify that the
braid relation holds on $B(w)$ for $w \in S_n$. This is done by
considering the following four sub-cases.
\begin{enumerate}
\item If $w \in S_{n-2}$ commutes with $s_{n-1}, s_n$, then both $T_n
T_{n-1} T_n \cdot B(w)$ and $T_{n-1} T_n T_{n-1} \cdot B(w)$ are easily
seen to equal $B(s_{n-1} w, 1, n-1)$.

\item Suppose $w = w' s_{n-2} \cdots s_{m'} \in S_{n-1} \setminus
S_{n-2}$, with $w' \in S_{n-2}$ and $m' \in [1,n-2]$. Then using Remark
\ref{Rsymm} and the $R_{n-1}$-module structure of $\scrm$, we compute:
\begin{align*}
T_n T_{n-1} T_n \cdot B(w) = &\ T_n T_{n-1} \cdot B(w,1,n) = T_n \cdot
B(w' s_{n-1} \cdots s_{m'}, 1,n)\\
= &\ B(w' s_{n-1}, 1, m'),\\
T_{n-1} T_n T_{n-1} \cdot B(w) = &\ T_{n-1} T_n \cdot B(w' s_{n-1} \cdots
s_{m'}) = T_{n-1} \cdot B(w',1,m')\\
= &\ B(s_{n-1} w',1,m'),
\end{align*}

\noindent whence we are done since $s_{n-1}$ commutes with $w' \in
S_{n-2}$.

\item In the last two sub-cases, $w = w' s_{n-1} \cdots s_{m'} \in S_n
\setminus S_{n-1}$ with $m' \in [1,n-1]$. As in the previous two
sub-cases, first suppose $w' \in S_{n-2}$. Similar to the above
computations, one verifies that both $T_n T_{n-1} T_n \cdot B(w)$ and
$T_{n-1} T_n T_{n-1} \cdot B(w)$ vanish.

\item Finally, suppose $w = w'' s_{n-2} \cdots s_m \cdot s_{n-1} \cdots
s_{m'}$ with $m \in [1,n-2], m' \in [1,n-1]$, and $w'' \in S_{n-2}$. Then
one verifies that
\[
T_n T_{n-1} T_n \cdot B(w) =
{\bf 1}(m < m') B(w'' s_{n-2} \cdots s_{m'-1},1,m) = 
T_{n-1} T_n T_{n-1} \cdot B(w).
\]
\end{enumerate}

\item Next suppose $b \in \B_2$ is of the form $B(w,k,m)$ with $k
\geqslant 2$ and $w = w' s_{n-1} \cdots s_{m'} \in S_n \setminus
S_{n-1}$. Then $T_n \cdot B(w,k,m) = 0$ by definition, so $T_n T_{n-1}
T_n \cdot B(w,k,m) = 0$. To show that $T_{n-1} T_n T_{n-1}$ kills
$B(w,k,m)$, we consider two sub-cases. If $w' \in S_{n-2}$, then $T_{n-1}
\cdot B(w, k, m) = 0$ and we are done. Otherwise suppose $w' = w''
s_{n-2} \cdots s_{m''} \in S_{n-1} \setminus S_{n-2}$, with $m'' \in
[1,n-2]$. Now compute using Remark \ref{Rsymm} and the relations verified
above:
\begin{align*}
&\ T_{n-1} T_n T_{n-1} \cdot B(w,k,m)\\
= &\ T_{n-1} T_n \cdot B(w'' s_{n-1} \cdots s_{m''} s_{n-1} \cdots
s_{m'}, k, m)\\
= &\ {\bf 1}(m'' < m') T_{n-1} T_n \cdot B(w'' s_{n-2} \cdots s_{m'-1}
s_{n-1} \cdots s_{m''},k,m)\\
= &\ {\bf 1}(m'' < m') T_{n-1} \cdot 0 = 0,
\end{align*}

\noindent where the penultimate equality uses that $k=2$.

\item Finally, suppose $b \in \B_3$. By the analysis in the first case
above, we only need to consider $b = B(w,1,m)$ with $w = w' s_{n-1}
\cdots s_{m'} \in S_n \setminus S_{n-1}$. It is now not hard to show that
both $T_n T_{n-1} T_n \cdot B(w,1,m)$ and $T_{n-1} T_n T_{n-1} \cdot
B(w,1,m)$ vanish if $w' \in S_{n-2}$. On the other hand, if $w' = w''
s_{n-2} \cdots s_{m''} \in S_{n-1} \setminus S_{n-2}$, then repeated use
of Remark \ref{Rsymm} (and equation \eqref{EnilcoxA}) shows that
\begin{align*}
&\ T_{n-1} T_n T_{n-1} \cdot B(w,1,m)\\
= &\ T_{n-1} T_n T_{n-1} \cdot B(w'' \cdot s_{n-2} \cdots s_{m''} \cdot
s_{n-1} \cdots s_{m'},1,m)\\
= &\ {\bf 1}(m'' < m') T_{n-1} T_n \cdot B(w'' \cdot s_{n-2} \cdots s_{m'
- 1} \cdot s_{n-1} \cdots s_{m''},1,m)\\
= &\ {\bf 1}(m'' < m') {\bf 1}(m'' < m) T_{n-1} \cdot B(w'' \cdot s_{n-2}
\cdots s_{m' - 1} \cdot s_{n-1} \cdots s_{m-1},1,m'')\\
= &\ {\bf 1}(m'' < m') {\bf 1}(m'' < m) {\bf 1}(m' < m) B(w'' s_{n-2}
\cdots s_{m-2} s_{n-1} \cdots s_{m'-1},1,m)\\
= &\ {\bf 1}(m'' < m' < m) B(w'' s_{n-2} \cdots s_{m-2} s_{n-1} \cdots
s_{m'-1},1,m).
\end{align*}

\noindent Notice this calculation shows the ``braid-like'' action of
$T_n, T_{n-1}$ on strings of the type
\[
T_{n-2} \cdots T_{m''}, \qquad
T_{n-1} \cdots T_{m'}, \qquad
T_n \cdots T_m.
\]

\noindent Similarly, one shows that
\[
T_n T_{n-1} T_n \cdot B(w,1,m) = {\bf 1}(m'' < m' < m) B(w'' s_{n-2}
\cdots s_{m-2} s_{n-1} \cdots s_{m'-1},1,m),
\]

\noindent which verifies that the last braid relation holds in the last
case.
\end{enumerate}

Thus the algebra relations hold on all of $\scrm$, making it an
$NC_A(n,d)$-module generated by $B(1)$, as claimed. In particular,
$NC_A(n,d) \simeq \scrm$ as $\bk$-modules, by the analysis in the first
part of this proof. This completes the proof of all but the last
assertion in Theorem \ref{ThmA}.
Finally, the nil-Coxeter algebra $NC_{A_l}((2,\dots,2))$ surjects onto
$R_l$, and $R_l \simeq R_l \cdot B(\emptyset) \subset V_1$ is free of
$\bk$-rank $(l+1)!$ from above. Hence $R_l \simeq NC_{A_l}((2,\dots,2))$,
as desired. \qed

\section{Proof of Theorem \ref{ThmB}, primitive elements, and
categorification}

In this section we continue our study of the algebras $NC_A(n,d)$,
starting with Theorem \ref{ThmB}.

\begin{proof}[Proof of Theorem \ref{ThmB}]
We retain the notation of Theorem \ref{ThmA}. Via the $\bk$-module
isomorphism $\scrm \simeq NC_A(n,d)$, we identify the basis element
$B(w,k,m)$ with $T_w T_n^k T_{n-1} \cdots T_m$ and $B(w)$ with $T_w$,
where $w \in S_n$, $k \in [1,d-1]$, and $m \in [1,n]$. Let $\ell : \B \to
\mathbb{Z}^{\geqslant 0}$ be as in equation \eqref{Elength}.

We now \textbf{claim} that if $\T = T_{i_1} \cdots T_{i_l}$ is
\textit{any} nonzero word in $NC_A(n,d)$, then $l$ is precisely the
length of $\T$ when expressed (uniquely) in standard form
\eqref{Estandard}. The proof is by induction on $l$. For $l=1$, $T_i$ is
already in standard form (and nonzero). Now given a word $\T = T_i \T'$
of length $l+1$ (so $\T'$ has length $l$ and satisfies the claim), write
$\T'$ via the induction hypothesis as a word in standard form of length
$l$. Now the proof of Theorem \ref{ThmA} shows that applying any $T_i$ to
this standard form for $\T'$ either yields zero or has length precisely
$l+1$. This proves the claim.

The above analysis shows (1) and (2). Now suppose $\bk$ is a field. Then
the algebra $NC_A(n,d)$ has a maximal ideal $\m = \tangle{ \{ T_i : i \in
I \} }$; in fact, $\m$ has $\bk$-corank $1$ by the proof of Theorem
\ref{ThmA}. Moreover, $\m$ is local because any element of $A \setminus
\m$ is invertible. (In particular, one understands representations of the
algebra $NC_A(n,d)$, e.g. by \cite[\S 6.1]{Kh}.)

The aforementioned claim also proves that $\m^{l+1} = 0$, where $l :=
\ell_{A_{n-1}}(w'_\circ) + d + n - 2$. This is because any nonzero word
can be expressed in standard form without changing the length.
\end{proof}

As an immediate consequence, we have:

\begin{cor}\label{Chilb}
If $\bk$ is a field and $T_1, \dots, T_n$ all have graded degree $1$, the
Hilbert--Poincar\'e series of $NC_A(n,d)$ is the polynomial
\[
[n]_q! \; (1 + [n]_q \; [d-1]_q), \qquad \text{where } [n]_q :=
\frac{q^n-1}{q-1}, \ [n]_q! := \prod_{j=1}^n [j]_q.
\]
\end{cor}

The proof also uses the standard result that the Hilbert--Poincar\'e
series of the usual nil-Coxeter algebra $NC_A(n,2)$ is $[n]_q!$ (see
e.g.~\cite[\S 3.12, 3.15]{Hum}).\smallskip

Next, we discuss a property that was explored in \cite{Kho} for the usual
nil-Coxeter algebras $NC_A(n,2)$: these algebras are always Frobenius. We
now study when the algebras $NC_A(n,d)$ are also Frobenius for $d
\geqslant 3$. As the following result shows, this only happens in the
degenerate case of $n=1$, i.e., $\bk[T_1] / (T_1^d)$.

\begin{theorem}\label{Pfrob}
Suppose $\bk$ is a field. Given $n \geqslant 1$ and $d \geqslant 2$, the
algebra $NC_A(n,d)$ is Frobenius if and only if $n=1$ or $d=2$.
\end{theorem}

One checks via equation \eqref{Edim} that these conditions are further
equivalent to (the group algebra of) the ``generalized Coxeter group''
$W(M_{n,d})$ being a flat deformation of $NC_A(n,d)$.

The proof of Theorem \ref{Pfrob} crucially uses the knowledge of
``maximal'', i.e., primitive words in the algebra $NC_A(n,d)$. Formally,
given a generalized Coxeter matrix $M$, say that an element $x \in NC(M)$
is \textit{left} (respectively, \textit{right}) \textit{primitive} if $\m
x = 0$ (respectively, $x \m = 0$), cf.~Theorem \ref{ThmB}(3). Now $x$ is
\textit{primitive} if it is both left- and right-primitive. Denote these
sets of elements respectively by
\[
\Prim_L(NC(M)), \quad \Prim_R(NC(M)), \quad \Prim(NC(M)).
\]

\begin{prop}\label{Pprim}
Every generalized nil-Coxeter algebra $NC(M)$ is equipped with an
anti-involution $\theta$ that fixes each generator $T_i$. Now $\theta$ is
an isomorphism $: \Prim_L(NC(M)) \longleftrightarrow \Prim_R(NC(M))$.
Moreover, the following hold.
\begin{enumerate}
\item If $W(M)$ is a finite Coxeter group with unique longest word
$w_\circ$, then
\[
\Prim_L(NC(M)) = \Prim_R(NC(M)) = \Prim(NC(M)) = \bk T_{w_\circ}.
\]

\item If $NC(M) = NC_A(1,d)$, then
\[
\Prim_L(NC(M)) = \Prim_R(NC(M)) = \Prim(NC(M)) = \bk \cdot T_1^{d-1}.
\]

\item If $NC(M) = NC_A(n,d)$ with $n \geqslant 2$ and $d \geqslant 3$,
then:
\begin{enumerate}
\item $\Prim_L(NC(M))$ is spanned by $T_{w_\circ} := T_{w'_\circ} T_n
T_{n-1} \cdots T_1$ and the $n(d-2)$ words
\[
\{ T_{w'_\circ} T_n^k T_{n-1} \cdots T_m : \ k \in [2, d-1], \ m \in
[1,n] \}.
\]

\item $\Prim(NC(M))$ is spanned by the $d-1$ words $T_{w'_\circ} T_n^k
T_{n-1} \cdots T_1$, where $1 \leqslant k \leqslant d-1$.
\end{enumerate}
\end{enumerate}

\noindent In all cases, the map $\theta$ fixes both $\Prim(NC(M))$ as
well as the lengths of all nonzero words.
\end{prop}

\begin{proof}
The first two statements are obvious since $\theta$ preserves the
defining relations in $\bk \tangle{ \{T_i : i \in I \} }$. The assertion
in (1) is standard -- see e.g. \cite[Chapter 7]{Hum} -- and (2) is easily
verified.

We next classify the left-primitive elements as in (3)(a). Suppose for
some $k \geqslant 2$ and $1 \leqslant m \leqslant n$ that $\T =
T_{w'_\circ} T_n^k T_{n-1} \cdots T_m$. Then clearly $T_i \T = 0$ for all
$i<n$, and $T_n \T = 0$ since $k \geqslant 2$, as discussed in the proof
of Theorem \ref{ThmA}. Similarly, if $\T = T_{w_\circ}$ then $T_i \T = 0$
for $i<n$, and we also computed in the proof of Theorem \ref{ThmA} that
$T_{w_\circ} = T_1 \cdots T_{n-1} T_n T_{w'_\circ}$. Hence,
\[
T_n T_{w_\circ} = T_1 \cdots T_{n-2} (T_n T_{n-1} T_n) T_{w'_\circ} = T_1
\cdots T_{n-1} T_n (T_{n-1} T_{w'_\circ}) = 0.
\]

\noindent To complete the proof of (3)(a), it suffices to show that no
nonzero linear combination of the remaining words of the form $T_w T_n^k
T_{n-1} \cdots T_m$ is left-primitive. Suppose first that there is a word
$w \in W_{A_{n-1}}$ such that the coefficient of $T_w$ is nonzero. In
that case, choose such an element $w$ of smallest length, and
left-multiply the linear combination by $T_n^{d-1} T_{w'_\circ w^{-1}}$.
As discussed in the proof of Theorem \ref{ThmA}, this kills all terms
$T_{w'} T_n^k T_{w''}$ with $w',w'' \in W_{A_{n-1}}$ and $k \geqslant 1$.
Moreover, by \cite[Chapter 7]{Hum}, left-multiplication by $T_{w'_\circ
w^{-1}}$ also kills all terms of the same length that are not $T_w$. Thus
we are left with $T_n^{d-1} T_{w'_\circ w^{-1}} T_w = T_n^{d-1}
T_{w'_\circ} \neq 0$, so the linear combination was not left-primitive.

The other case is that all words in the linear combination are of the
form $T_w T_n^k T_{n-1} \cdots T_m$ with $k \geqslant 1$. Once again,
choose $w \in W_{A_{n-1}}$ of smallest length for which the corresponding
word has nonzero coefficient, and left-multiply by $T_{w'_\circ w^{-1}}$.
This yields a nonzero linear combination by the analysis in Theorem
\ref{ThmA}, which proves the assertion about left-primitivity.

We next identify the primitive elements in $NC(M) = NC_A(n,d)$. The first
claim is that $\T_k := T_{w'_\circ} T_n^k T_{n-1} \cdots T_1$ is fixed by
$\theta$. Indeed, we compute using the braid relations in type $A$ that
$\theta$ fixes $T_{w'_\circ} \in R_{n-1}$ and $T_{w''_\circ} \in
R_{n-2}$. Hence,
\begin{equation}\label{Eprim}
\theta(\T_k) = T_1 \cdots T_{n-1} T_n^k T_{w''_\circ} T_{n-1} \cdots T_1
= T_1 \cdots T_{n-1} T_{w''_\circ} T_n^k T_{n-1} \cdots T_1 = \T_k.
\end{equation}

\noindent Using this we claim that $\T_k$ is right-primitive.
Indeed, if $i<n$, then
\[
\T_k T_i = T_1 \cdots T_{n-1} T_n^k T_{w'_\circ} T_i = 0,
\]

\noindent while for $i=n$, we compute:
\[
\T_k T_n = T_{w'_\circ} T_n^k T_{n-1} T_n \cdot T_{n-2} \cdots T_1 =
T_n^{k-1} T_{w'_\circ} T_{n-1} \cdot T_n T_{n-1} T_{n-2} \cdots T_1 = 0.
\]

\noindent We now claim that no linear combination of the remaining
left-primitive elements listed in (3)(a) is right-primitive. Indeed, let
$m_0$ denote the minimum of the $m$-values in words with nonzero
coefficients; then $m_0 > 1$ by the above analysis. Now right-multiply by
$T_{m_0 - 1} \cdots T_1$. This kills all elements with $m$-value $> m_0 +
1$, since $T_{m_0 - 1}$ commutes with $T_n T_{n-1} \cdots T_m$, hence can
be taken past them to multiply against $T_{w'_\circ}$ and be killed. The
terms with $m$-value equal to $m_0$ are not killed, by the analysis in
Theorem \ref{ThmA}. It follows that such a linear combination is not
right-primitive, which completes the classification of the primitive
elements in (3)(b).

Next, that $\Prim(NC(M))$ is fixed by $\theta$ was shown in equation
\eqref{Eprim}. Moreover, if $NC(M)$ equals $NC_A(n,d)$ or $\bk W(M)$ with
$W(M)$ finite, then it is equipped with a suitable length function
$\ell$. Now $\theta$ preserves the length because the algebra relations
are $\ell$-homogeneous and preserved by $\theta$.
\end{proof}

\begin{remark}\label{Rprim}
In light of Proposition \ref{Pprim}, it is natural to ask how to write
right-primitive words in standard form. More generally, given $w = w'
s_{n-1} \cdots s_{m'}$ for unique $w' \in S_{n-1}$ and $m' \in [1,n]$
(via Lemma \ref{Lsymm}), we have:
$T_m \cdots T_{n-1} T_n^k T_w = T_{\widetilde{w}} T_n^k T_{n-1} \cdots
T_{m'}$, where $\widetilde{w} = s_m \cdots s_{n-1} w'$.
\end{remark}

With Proposition \ref{Pprim} in hand, we turn to the Frobenius property
of $NC_A(n,d)$. The following proof reveals that $NC_A(n,d)$ is Frobenius
if and only if $\Prim(NC(M))$ is one-dimensional.

\begin{proof}[Proof of Theorem \ref{Pfrob}]
For finite Coxeter groups $W(M)$, the corresponding nil-Coxeter algebras
$NC(M)$ are indeed Frobenius; see e.g. \cite[\S 2.2]{Kho}. It is also
easy to verify that $NC_A(1,d) = \bk[T_1] / (T_1^d)$ is Frobenius, by
using the symmetric bilinear form uniquely specified by: $\sigma(T_1^i,
T_1^j) = {\bf 1}(i+j=d-1)$. Thus, it remains to show that for $n
\geqslant 2$ and $d \geqslant 3$, the algebra $NC_A(n,d)$ is not
Frobenius. Indeed, if $NC_A(n,d)$ is Frobenius with nondegenerate
invariant bilinear form $\sigma$, then for each nonzero primitive $p$
there exists a vector $a_p$ such that $0 \neq \sigma(p,a_p) = \sigma(p
a_p, 1)$. It follows that we may take $a_p = 1$ for all $p$. Now the
linear functional $\sigma(-,1) : \Prim(NC_A(n,d)) \to \bk$ is
nonsingular, whence $\dim_\bk \Prim(NC_A(n,d)) = 1$. Thus $n=1$ or $d=2$
by Proposition \ref{Pprim}.
\end{proof}

We conclude this section by discussing the connection of $NC_A(n,d)$ to
the categorification by Khovanov \cite{Kho} of the Weyl algebra $W_n :=
\mathbb{Z} \langle x,\partial \rangle / (\partial x = 1 + x \partial)$.
Namely, the usual type $A$ nil-Coxeter algebra $\mathcal{A}_n :=
NC_A(n,2)$ is a bimodule over $\mathcal{A}_{n-1}$, and this structure was
studied in \textit{loc.~cit.}, leading to the construction of tensor
functors categorifying the operators $x, \partial$.

We now explain how the algebra $NC_A(n,d)$ fits into this framework.

\begin{prop}\label{Pkhovanov}
For all $n \geqslant 1$ and $d \geqslant 2$, we have an isomorphism of
$\mathcal{A}_{n-1}$-bimodules:
\[
NC_A(n,d) \simeq \mathcal{A}_{n-1} \oplus \bigoplus_{k=1}^{d-1} \left(
\mathcal{A}_{n-1} \otimes_{\mathcal{A}_{n-2}} \mathcal{A}_{n-1} \right).
\]
\end{prop}

When $d=2$, this result was shown in \cite[Proposition 5]{Kho}. For
general $d \geqslant 2$, using the notation of \cite{Kho}, this result
implies in the category of $\mathcal{A}_{n-1}$-bimodules that the algebra
$NC_A(n,d)$ corresponds to $1 + (d-1) x \partial$ (including the
previously known case of $d=2$). In particular, Proposition
\ref{Pkhovanov} strengthens Theorems \ref{ThmA} and \ref{ThmB}, which
explained a left $\mathcal{A}_{n-1}$-module structure on $NC_A(n,d)$
(namely, that $NC_A(n,d)$ is free of rank $1 + n(d-1)$).

\begin{proof}[Proof of Proposition \ref{Pkhovanov}]
From the proof of Theorem \ref{ThmA}, the algebra $NC_A(n,d)$ has a
`regular representation' $\varphi : \scrm \overset{\sim}{\longrightarrow}
NC_A(n,d)$, sending $B(w) \mapsto T_w$ and $B(w,k,m) \mapsto T_w T_n^k
T_{n-1} \cdots T_m$ for $w \in S_n, \ k \in [1,d-1]$, and $m \in [1,n]$.
Also recall the subspaces $V_{k,m} $ defined in the discussion following
equation \eqref{Estandard}: $V_{k,m} = \bigoplus_{w \in S_n} \bk
B(w,k,m)$.

By Theorem \ref{ThmA}, $M_k := \bigoplus_{m=1}^n \varphi(V_{k,m})$ is a
free left $\mathcal{A}_{n-1}$-module of rank one. It is also a free right
$\mathcal{A}_{n-1}$-module of rank one, using the anti-involution
$\theta$ from Proposition \ref{Pprim} and Remark \ref{Rprim}. In fact,
the uniqueness of the standard form \eqref{Estandard} shown in the proof
of Theorem \ref{ThmA}, implies that for all $1 \leqslant k \leqslant
d-1$, the map
\[
\varphi_k : \mathcal{A}_{n-1} \otimes_{\mathcal{A}_{n-2}}
\mathcal{A}_{n-1} \to M_k, \qquad a \otimes a' \mapsto a T_n^k a',
\]

\noindent is an isomorphism of $\mathcal{A}_{n-1}$-bimodules. Now the
result follows from (the proof of) Theorem \ref{ThmA}.
\end{proof}

\begin{remark}
Notice that the proof of Proposition \ref{Pkhovanov} also categorifies
Corollary \ref{Chilb}.
\end{remark}

\section{Proof of Theorem \ref{ThmC}: Finite-dimensional
generalized nil-Coxeter algebras}

We now prove Theorem \ref{ThmC}, which classifies the generalized
nil-Coxeter algebras of finite $\bk$-rank.
The bulk of the proof involves showing $(1) \implies (2)$. We again
employ the diagrammatic calculus used to show Theorem \ref{ThmA}, now
applied to the five diagrams in Figure \ref{Fig3} below.\medskip

\begin{figure}[ht]
\begin{tikzpicture}[line cap=round,line join=round,>=triangle 45,x=1.0cm,y=1.0cm]
\draw(3,11.5) circle (0.25cm);
\draw(5,11.5) circle (0.25cm);
\draw(5,9.5) circle (0.25cm);
\draw(3,9.5) circle (0.25cm);
\draw (2.7,11.75) node[anchor=north west] {\textit{A}};
\draw (4.7,11.75) node[anchor=north west] {\textit{B}};
\draw (4.7,9.75) node[anchor=north west] {\textit{C}};
\draw (2.7,9.75) node[anchor=north west] {\textit{D}};
\draw [->] (3,9.9) -- (3,11.1);
\draw [->] (3.4,11.5) -- (4.6,11.5);
\draw [->] (5,11.1) -- (5,9.9);
\draw [->] (4.6,9.5) -- (3.4,9.5);
\draw (2.5,10.6) node[anchor=north west] {$ \alpha $};
\draw (3.7,12) node[anchor=north west] {$ \alpha $};
\draw (3.8,10.1) node[anchor=north west] {$ \gamma $};
\draw (5,10.9) node[anchor=north west] {$ \gamma $};
\draw (3.1,11.4) node[anchor=north west] {+};
\draw (3.4,8.9) node[anchor=north west] {Fig. 3.1}; 
\draw(5.8,13) circle (0.25cm);
\draw(7.5,14.1) circle (0.4cm);
\draw(9.5,14.1) circle (0.4cm);
\draw(11.5,14.1) circle (0.4cm);
\draw(13.5,14.1) circle (0.4cm);
\draw(13.5,11.9) circle (0.4cm);
\draw(11.5,11.9) circle (0.4cm);
\draw(9.5,11.9) circle (0.4cm);
\draw(7.5,11.9) circle (0.4cm);
\draw(14.9,13) circle (0.25cm);
\draw (5.5,13.3) node[anchor=north west] {$A$};
\draw (7.1,14.4) node[anchor=north west] {$B_1$};
\draw (9.1,14.4) node[anchor=north west] {$B_2$};
\draw (11,14.4) node[anchor=north west] {$B_{m'}$};
\draw (13.05,14.4) node[anchor=north west] {$B_m$};
\draw (13,12.2) node[anchor=north west] {$B'_m$};
\draw (11,12.2) node[anchor=north west] {$B'_{m'}$};
\draw (9.1,12.2) node[anchor=north west] {$B'_2$};
\draw (7.1,12.2) node[anchor=north west] {$B'_1$};
\draw (14.6,13.3) node[anchor=north west] {$C$};
\draw [->] (7,12.2) -- (6.1,12.8);
\draw [->] (6.1,13.2) -- (7,14);
\draw [->] (8,14.1) -- (9,14.1);
\draw [->] (12,14.1) -- (13,14.1);
\draw [->] (13,11.9) -- (12,11.9);
\draw [->] (9,11.9) -- (8,11.9);
\draw [->] (13.9,13.8) -- (14.6,13.2);
\draw [->] (14.6,12.8) -- (13.9,12.1);
\draw (6.2,12.5) node[anchor=north west] {$ \alpha $};
\draw (6.2,14) node[anchor=north west] {$ \alpha $};
\draw (8.1,14.7) node[anchor=north west] {$ \beta_1 $};
\draw (10.1,14.3) node[anchor=north west] {$ \cdots $};
\draw (12,14.7) node[anchor=north west] {$ \beta_{m'} $};
\draw (14.2,14.1) node[anchor=north west] {$ \gamma $};
\draw (14.2,12.5) node[anchor=north west] {$ \gamma $};
\draw (12.3,11.9) node[anchor=north west] {$ \beta_{m'} $};
\draw (10.1,12.1) node[anchor=north west] {$ \cdots $};
\draw (8.3,11.9) node[anchor=north west] {$ \beta_1 $};
\draw (6,13.3) node[anchor=north west] {+};
\draw (8.7,11.1) node[anchor=north west] {Fig. 3.2 ($m' = m-1$)};
\draw(3,5.7) circle (0.25cm);
\draw(5,5.7) circle (0.25cm);
\draw(3,3.7) circle (0.25cm);
\draw (2.7,5.95) node[anchor=north west] {\textit{A}};
\draw (4.7,5.95) node[anchor=north west] {\textit{B}};
\draw (2.7,3.95) node[anchor=north west] {\textit{C}};
\draw [->] (3,4.1) -- (3,5.3);
\draw [->] (3.4,5.7) -- (4.6,5.7);
\draw [->] (4.7,5.4) -- (3.3,4);
\draw (2.5,4.9) node[anchor=north west] {$ t $};
\draw (3.7,6.2) node[anchor=north west] {$ s $};
\draw (4,4.8) node[anchor=north west] {$ u $};
\draw (3.1,5.6) node[anchor=north west] {+};
\draw (3.4,3.6) node[anchor=north west] {Fig. 3.3}; 
\draw(5.8,7.7) circle (0.25cm);
\draw(7.5,8.8) circle (0.4cm);
\draw(9.5,8.8) circle (0.4cm);
\draw(11.5,8.8) circle (0.4cm);
\draw(13.5,8.8) circle (0.4cm);
\draw(13.5,6.6) circle (0.4cm);
\draw(11.5,6.6) circle (0.4cm);
\draw(9.5,6.6) circle (0.4cm);
\draw(7.5,6.6) circle (0.4cm);
\draw(12.1,7.7) circle (0.25cm);
\draw(14.9,7.7) circle (0.25cm);
\draw (5.5,8) node[anchor=north west] {$A$};
\draw (7.1,9.1) node[anchor=north west] {$B_1$};
\draw (9.1,9.1) node[anchor=north west] {$B_2$};
\draw (11,9.1) node[anchor=north west] {$B_{m'}$};
\draw (13.05,9.1) node[anchor=north west] {$B_m$};
\draw (13,6.9) node[anchor=north west] {$B'_m$};
\draw (11,6.9) node[anchor=north west] {$B'_{m'}$};
\draw (9.1,6.9) node[anchor=north west] {$B'_2$};
\draw (7.1,6.9) node[anchor=north west] {$B'_1$};
\draw (11.8,8) node[anchor=north west] {$D$};
\draw (14.6,8) node[anchor=north west] {$C$};
\draw [->] (7,6.9) -- (6.1,7.5);
\draw [->] (6.1,7.9) -- (7,8.7);
\draw [->] (8,8.8) -- (9,8.8);
\draw [->] (12,8.8) -- (13,8.8);
\draw [->] (13,6.6) -- (12,6.6);
\draw [->] (9,6.6) -- (8,6.6);
\draw [->] (13.1,8.6) -- (12.3,7.9);
\draw [->] (12.3,7.5) -- (13.1,6.9);
\draw [->] (13.9,8.5) -- (14.6,7.9);
\draw [->] (14.6,7.5) -- (13.9,6.8);
\draw (6.2,7.2) node[anchor=north west] {$ \alpha $};
\draw (6.2,8.7) node[anchor=north west] {$ \alpha $};
\draw (8.1,9.4) node[anchor=north west] {$ \beta_1 $};
\draw (10.1,9) node[anchor=north west] {$ \cdots $};
\draw (12,9.4) node[anchor=north west] {$ \beta_{m'} $};
\draw (12.7,8.4) node[anchor=north west] {$ \gamma $};
\draw (12.7,7.6) node[anchor=north west] {$ \delta $};
\draw (14.2,8.8) node[anchor=north west] {$ \delta $};
\draw (14.2,7.2) node[anchor=north west] {$ \gamma $};
\draw (12.3,6.6) node[anchor=north west] {$ \beta_{m'} $};
\draw (10.1,6.8) node[anchor=north west] {$ \cdots $};
\draw (8.3,6.6) node[anchor=north west] {$ \beta_1 $};
\draw (6,8) node[anchor=north west] {+};
\draw (8.7,5.8) node[anchor=north west] {Fig. 3.4 ($m' = m-1$)};
\draw(5.8,2.1) circle (0.25cm);
\draw(7.5,3.2) circle (0.4cm);
\draw(9.5,3.2) circle (0.4cm);
\draw(11.5,3.2) circle (0.4cm);
\draw(13.5,3.2) circle (0.4cm);
\draw(13.5,1) circle (0.4cm);
\draw(11.5,1) circle (0.4cm);
\draw(9.5,1) circle (0.4cm);
\draw(7.5,1) circle (0.4cm);
\draw (5.5,2.4) node[anchor=north west] {$A$};
\draw (7.1,3.5) node[anchor=north west] {$B_1$};
\draw (9.1,3.5) node[anchor=north west] {$B_2$};
\draw (11,3.5) node[anchor=north west] {$B_{m'}$};
\draw (13.05,3.5) node[anchor=north west] {$B_m$};
\draw (13,1.3) node[anchor=north west] {$B'_m$};
\draw (11,1.3) node[anchor=north west] {$B'_{m'}$};
\draw (9.1,1.3) node[anchor=north west] {$B'_2$};
\draw (7.1,1.3) node[anchor=north west] {$B'_1$};
\draw [->] (7,1.3) -- (6.1,1.9);
\draw [->] (6.1,2.3) -- (7,3.1);
\draw [->] (8,3.2) -- (9,3.2);
\draw [->] (12,3.2) -- (13,3.2);
\draw [->] (13.5,2.7) -- (13.5,1.5);
\draw [->] (13,1) -- (12,1);
\draw [->] (9,1) -- (8,1);
\draw (6.2,1.6) node[anchor=north west] {$ \alpha $};
\draw (6.2,3.1) node[anchor=north west] {$ \alpha $};
\draw (8.1,3.8) node[anchor=north west] {$ \beta_1 $};
\draw (10.1,3.4) node[anchor=north west] {$ \cdots $};
\draw (12,3.8) node[anchor=north west] {$ \beta_{m'} $};
\draw (13.6,2.4) node[anchor=north west] {$ \gamma $};
\draw (12.3,1) node[anchor=north west] {$ \beta_{m'} $};
\draw (10.1,1.2) node[anchor=north west] {$ \cdots $};
\draw (8.3,1) node[anchor=north west] {$ \beta_1 $};
\draw (6,2.4) node[anchor=north west] {+};
\draw (8.7,0) node[anchor=north west] {Fig. 3.5 ($m' = m-1$)};
\end{tikzpicture}
\caption{Modules for the infinite-dimensional generalized nil-Coxeter
algebras}
\label{Fig3}
\end{figure}

We begin by \textbf{assuming} that $W = W(M)$ is a generalized Coxeter
group, and classify the algebras $NC(M)$ that have finite $\bk$-rank.
Following this classification, we address the remaining finite complex
reflection groups $W$ (and all ${\bf d}$), followed by the infinite
discrete complex reflection groups with their Coxeter-type
presentations.\medskip

\noindent \textit{Case $1$.}
Suppose $m_{ii} = 2$ for all $i \in I$. In this case $W(M)$ is a Coxeter
group, so by e.g. \cite[Chapter 7]{Hum}, $NC(M)$ has a $\bk$-basis in
bijection with $W(M)$, which must therefore be finite.\medskip

\noindent \textit{Case $2$.}
Suppose $m_{\alpha \alpha}, m_{\gamma \gamma} \geqslant 3$ for some
$\alpha, \gamma \in I$ with $m_{\alpha \gamma} \geqslant 3$. In this case
we appeal to Figure 3.1 and work as in the proof of \cite[Proposition
3.2]{Ma}. Thus, fix a free $\bk$-module $\scrm$ with basis given by the
countable set $\{ A_r, B_r, C_r, D_r : r \geqslant 1 \}$, and define an
$NC(M)$-action via Figure 3.1. Namely, $T_i$ kills all basis vectors for
all $i \in I$, with the following exceptions:
\[
T_\alpha(A_r) := B_r, \quad 
T_\gamma(B_r) := C_r, \quad 
T_\gamma(C_r) := D_r, \quad 
T_\alpha(D_r) := A_{r+1}, \quad \forall r \geqslant 1.
\]

\noindent (The ``+'' at the head of an arrow refers precisely to the
index increasing by $1$.) It is easy to verify that the defining
relations of $NC(M)$ hold in ${\rm End}_\bk(\scrm)$, as they hold on each
$A_r, B_r, C_r, D_r$. Therefore $\scrm$ is a module over $NC(M)$ that is
generated by $A_1$, but is not finitely generated as a $\bk$-module. As
$NC(M) \twoheadrightarrow \scrm$, $NC(M)$ is also not a finitely
generated $\bk$-module.\medskip

This approach is used in the remainder of the proof, to obtain a
$\bk$-basis and the $NC(M)$-action on it, from the diagrams in Figure
\ref{Fig3}. Thus we only mention the figure corresponding to each of the
cases below.\medskip

\noindent \textit{Case $3$.}
Figure 3.1 is actually a special case of Figure 3.2, and was included to
demonstrate a simpler case. Now suppose more generally that there are two
nodes $\alpha, \gamma \in I$ such that $m_{\alpha \alpha}, m_{\gamma
\gamma} \geqslant 3$. Since the Coxeter graph is connected, there exist
nodes $\beta_1, \dots, \beta_{m-1}$ for some $m \geqslant 1$ (in the
figure we write $m' := m-1$), such that
\[
\alpha \quad \longleftrightarrow \quad \beta_1 \quad \longleftrightarrow
\quad \cdots \quad \longleftrightarrow \quad \beta_{m-1} \quad
\longleftrightarrow \quad \gamma
\]

\noindent is a path (so each successive pair of nodes is connected by at
least a single edge).
Now appeal to Figure 3.2, i.e., define an $NC(M)$-module structure on the
free $\bk$-module
\[
\scrm := {\rm span}_\bk \{ A_r, B_{1r}, \dots, B_{mr}, C_r, B'_{1r},
\dots, B'_{mr} : r \geqslant 1 \},
\]

\noindent where each $T_i$ kills all basis vectors above, except for the
actions obtained from Figure 3.2. Once again, $\scrm$ is generated by
$A_1$, so proceed as above to show that $NC(M)$ is not finitely
generated.\medskip

\noindent \textit{Case $4$.}
The previous cases reduce the situation to a unique vertex $\alpha$ in
the Coxeter graph of $M$ for which $m_{\alpha \alpha} \geqslant 3$. The
next two steps show that $\alpha$ is adjacent to a unique node $\gamma$,
and that $m_{\alpha \gamma} = 3$. First suppose $\alpha$ is adjacent to
$\gamma$ with $m_{\alpha \gamma} \geqslant 4$. Now appeal to Figure 3.3,
setting $(s,t,u) \leadsto (\alpha, \alpha, \gamma)$, and define an
$NC(M)$-module structure on
$\scrm := {\rm span}_\bk \{ A_r, B_r, C_r : r \geqslant 1 \}$. Then
proceed as above.\medskip

\noindent \textit{Case $5$.}
Next suppose $\alpha$ is adjacent in the Coxeter graph to two nodes
$\gamma, \delta$. By the previous case, $m_{\alpha \gamma} = m_{\alpha
\delta} = 3$. Now appeal to Figure 3.4 with $m=1$, to define an
$NC(M)$-module structure on
$\scrm := {\rm span}_\bk \{ A_r, B_{1r}, B'_{1r}, C_r, D_r : r \geqslant
1 \}$,
and proceed as in the previous cases.\medskip

We now observe that if $NC(M)$ is finitely generated, then so is
$NC_M((2, \dots, 2))$, which corresponds to the Coxeter group $W_M((2,
\dots, 2))$. Hence the Coxeter graph of $M$ is of finite type. These
graphs were classified by Coxeter \cite{Cox,Cox2}. We now rule out all
cases other than type $A$, in which case the above analysis shows that
${\bf d} = (2, \dots, 2, d)$ or $(d, 2, \dots, 2)$.\medskip

\noindent \textit{Case $6$.}
First notice that dihedral types (i.e., types $G_2,H_2,I$) are already
ruled out by the above cases. The same cases also rule out one
possibility in types $B,C,H$, where we may now set $n \geqslant 3$. For
the remaining cases of types $B,C,H$, assume that the Coxeter graph is
labelled
\[
\alpha \quad \longleftrightarrow \quad \beta_1 \quad \longleftrightarrow
\quad \cdots \quad \longleftrightarrow \quad \beta_{m-1} \quad
\longleftrightarrow \quad \gamma,
\]

\noindent with $m_{\alpha \alpha} \geqslant 3, m_{\gamma \gamma} = 2,
m_{\beta_{m-1} \gamma} \geqslant 4$. In this case we construct the
$NC(M)$-module $\scrm$ by appealing to Figure 3.5; now proceed as
above.\medskip

\noindent \textit{Case $7$.}
The next case is of type $D_n$, with $n \geqslant 4$. Notice that
$\alpha$ is an extremal (i.e., pendant) vertex by the above analysis.
First assume $\alpha$ is the extremal node on the ``long arm'' of the
Coxeter graph. Now appeal to Figure 3.4 with $m = n-2$, to construct an
$NC(M)$-module $\scrm$.

The other sub-case is when $\alpha$ is one of the other two extremal
nodes in the $D_n$-graph. Define the quotient algebra $NC'(M)$ whose
Coxeter graph is of type $D_4$ (i.e., where we kill the $n-4$ generators
$T_i$ in the long arm that are the furthest away from $\alpha$). Now
repeat the construction in the previous paragraph, using Figure 3.4 with
$m=2$. It is easy to verify that the space $\scrm$ is a module for
$NC'(M)$, hence for the algebra $NC_{D_4}((2,2,2, m_{\alpha \alpha}))$.
This allows us to proceed as in the previous sub-case and show that
$NC'(M)$ is not finitely generated, whence neither is $NC(M)$.\medskip

\noindent \textit{Case $8$.}
If the Coxeter graph is of type $E$, then we may reduce to the $D_n$-case
by the analysis in the previous case. Hence it follows using Figure 3.4
that $NC(M)$ is not finitely generated.\medskip

\noindent \textit{Case $9$.}
If the Coxeter graph is of type $F_4$, then we may reduce to the
$B_n$-case by the analysis in Case 7. It now follows from Case 6 that
$NC(M)$ is not finitely generated.\medskip

This completes the classification for generalized Coxeter groups $W(M)$.
We now appeal to the classification and presentation of all finite
complex reflection groups, whose Coxeter graph is connected. These groups
and their presentations are listed in \cite[Tables 1--4]{BMR2}. In what
follows, we adopt the following notation: if $W = G_m$ for $4 \leqslant m
\leqslant 37$, then the corresponding generalized nil-Coxeter algebras
will be denoted by $NC_m({\bf d})$. Similarly if $W = G(de,e,r)$, then we
work with $NC_{(de,e,r)}({\bf d})$.
In what follows, we will often claim that $NC_W({\bf d})$ is not finitely
generated (over $\bk$), omitting the phrase ``unless it is the usual
nil-Coxeter algebra over a finite Coxeter group''.\medskip

\noindent \textit{Case $10$} (Exceptional types with finite Coxeter
graph).
If $W = G_m$ for $m = 4, 8, 16, 25, 32$, then its Coxeter graph is of
type $A$. This case has been addressed above; thus the only possibility
that $NC_W({\bf d})$ has finite rank is that it equals
$NC_{A_n}((2,\dots,2,d))$ for $d \geqslant 2$, as desired.

Next if $W = G_m$ for $m = 5, 10, 18, 26$, then its Coxeter graph is of
type $B$, which was also addressed above and never yields an algebra of
finite $\bk$-rank. Now suppose $W = G_{29}$. Then $s,t,u$ form a
sub-diagram of type $B_3$, whence the  quotient algebra $NC'_{29}$
generated by $T_s, T_t, T_u$ is not finitely generated, by arguments as
in Case 7 above. It follows that $NC_{29}({\bf d})$ is also not finitely
generated.

The next case is if $W = G_m$ for $m = 6, 9, 14, 17, 20, 21$. In this
case the Coxeter graph is of dihedral type, which was also addressed
above.\medskip

\noindent \textit{Case $11$} (All other exceptional types).
For the remaining exceptional values of $m \in [4,37]$, with $W$ not a
finite Coxeter group, we will appeal to Figure 3.3. There are three
cases: first, suppose $m = 31$. In this case, set $(s,t,u) \leadsto
(s,u,t)$ in Figure 3.3 and define an $NC_m({\bf d})$-module $\scrm$ that
is $\bk$-free with basis $\{ A_r, B_r, C_r : r \geqslant 1 \}$. Now
proceed as above.

Next if $m = 33,34$, then set $(s,t,u) \leadsto (w,t,u)$ in Figure 3.3 to
define an $NC_m({\bf d})$-module $\scrm$, and proceed as above.

Finally, fix any other $m$, i.e., $m = 7, 11, 12, 13, 15, 19, 22, 24,
27$. In this case, use Figure 3.3 to define an $NC_m({\bf d})$-module
$\scrm$, and proceed as above.\medskip

\noindent \textit{Case $12$} (The infinite families).
It remains to consider the six infinite families enumerated in
\cite{BMR2}, which make up the family $G(de,e,r)$. Three of the families
consist of finite Coxeter groups of types $A,B,I$, which were considered
above. We now consider the other three families.
\begin{enumerate}[(a)]
\item Suppose $W = G(de,e,r)$ with $e \geqslant 3$. Then by \cite[Table
1]{BMR2}, consider the quotient algebra $NC'_{(de,e,r)}$ of
$NC_{(de,e,r)}({\bf d})$ which is generated by $s, t := t'_2, u := t_2$,
by killing all other generators $T_i$. The generators of $NC'_{(de,e,r)}$
now satisfy the relations
\[
T_s^{d_s} = T_t^{d_t} = T_u^{d_u} = 0, \qquad T_s T_t T_u = T_t T_u T_s,
\qquad \underbrace{T_u T_s T_t T_u \cdots}_{(e+1)\ times} =
\underbrace{T_s T_t T_u T_s \cdots}_{(e+1)\ times}.
\]

\noindent Thus, use Figure 3.3 to define an $NC'_{(de,e,r)}$-module
structure on $\scrm$, and proceed as above to show that
$NC_{(de,e,r)}({\bf d})$ is not finitely generated.

\item Suppose $W = G(2d,2,r)$ with $d \geqslant 2$; see \cite[Table
2]{BMR2}. Apply a similar argument as in the previous sub-case, using the
same generators and the same figure.

\item Suppose $W = G(e,e,r)$ with $e \geqslant 2$ and $r > 2$. If $e=2$,
then $G(2,2,r)$ is a finite Coxeter group, hence was addressed above.
Next, if $r>3$ then killing $T_s$ reduces to (a quotient of) the
$D_n$-case, which was once again addressed above. Finally, suppose $r=3
\leqslant e$. Setting $s := t_3, t := t'_2, u := t_2$, the generators of
$NC_{(e,e,3)}$ satisfy
\begin{align*}
T_s^{d_s} = T_t^{d_t} = T_u^{d_u} = 0, & \qquad T_s T_t T_s = T_t T_s
T_t, \qquad T_s T_u T_s = T_u T_s T_u, \\
(T_s T_t T_u)^2 = (T_t T_u T_s)^2, & \qquad \underbrace{T_u T_t T_u T_t
\cdots}_{e\ times} = \underbrace{T_t T_u T_t T_u \cdots}_{e\ times}.
\end{align*}

\noindent Once again, use Figure 3.3 to define an $NC_{(e,e,3)}$-module
structure on $\scrm$, and proceed as above.
\end{enumerate}

This completes the proof of $(1) \implies (2)$ for finite complex
reflection groups. Next, by e.g. \cite[Chapter 7]{Hum}, for no infinite
Coxeter group $W$ is $NC_W((2,\dots,2))$ a finitely generated
$\bk$-module, whence the same result holds for $NC_W({\bf d})$ when all
$d_i \geqslant 2$. We now use the classification of the (remaining)
infinite complex reflection groups $W$ associated to a connected braid
diagram. These groups were described in \cite{Po1} and subsequently in
\cite{Mal}. Thus, there exists a complex affine space $E$ with group of
translations $V$; choosing a basepoint $v_0 \in E$, we can identify the
semidirect product $GL(V) \ltimes V$ with the group $A(E)$ of
\textit{affine transformations} of $E$. Moreover, $W \subset A(E)$.
Define $\Lin(W)$ to be the image of $W$ in the factor group $GL(V)$, and
$\Tran(W)$ to be the subset of $W$ in $V$, i.e.,
\begin{equation}
\Tran(W) := W \cap V, \qquad \Lin(W) := W / \Tran(W).
\end{equation}

\noindent It remains to consider three cases for irreducible infinite
complex reflection groups $W$.\medskip

\noindent \textit{Case $13$.}
The group $W$ is \textit{noncrystallographic}, i.e., $E/W$ is not
compact. Then by \cite[Theorem 2.2]{Po1}, there exists a real form
$E_{\mathbb{R}} \subset E$ whose complexification is $E$, i.e.,
$E_{\mathbb{R}} \otimes_{\mathbb{R}} \mathbb{C} = E$. Moreover, by the
same theorem, restricting the elements of $W$ to $E_{\mathbb{R}}$ yields
an affine Weyl group $W_{\mathbb{R}}$. Hence if $NC_W({\bf d})$ is a
finitely generated $\bk$-module, then so is $NC_{W_{\mathbb{R}
}}((2,\dots,2))$, which is impossible.\medskip

\noindent \textit{Case $14$.}
The group $W$ is a \textit{genuine crystallographic group}, i.e., $E/W$
is compact and $\Lin(W)$ is not the complexification of a real reflection
group. Such groups were studied by Malle in \cite{Mal}, and Coxeter-type
presentations for these groups were provided in Tables I, II in
\textit{loc.~cit.}
Specifically, Malle showed that these groups are quotients of a free
monoid by a set of braid relations and order relations, together with one
additional order relation $R_0^{m_0} = 1$. We now show that for none of
these groups $W$ is the algebra $NC_W({\bf d})$ (defined in Definition
\ref{Dfinite}) a finitely generated $\bk$-module. To do so, we proceed as
above, by specifying the sub-figure in Figure \ref{Fig3} that corresponds
to each of these groups. There are three sub-cases:
\begin{enumerate}
\item Suppose $W$ is the group $[G(3,1,1)]$ in \cite[Table I]{Mal}, or
$[K_m]$ in \cite[Table II]{Mal} for $m=4,8,25,32$. For these groups we
appeal to Figure 3.1 and proceed as in Case $2$ above.

\item For $W = [K_{33}], [K_{34}]$, notice that it suffices to show the
\textbf{claim} that given the $\widetilde{A_3}$ Coxeter graph (i.e., a
$4$-cycle) with nodes labelled $\alpha_1, \dots, \alpha_4$ in clockwise
fashion, the corresponding algebra $NC_{\widetilde{A_3}}({\bf d})$ is not
a finitely generated $\bk$-module. For this we construct a module
$\scrm$, using Figure \ref{Fig4} below with $n=4$.
Now proceeding as above shows the claim, and hence the result for
$[K_{33}], [K_{34}]$.

\begin{figure}[ht]
\begin{tikzpicture}[line cap=round,line join=round,>=triangle 45,x=1.0cm,y=1.0cm]
\draw(0.5,3) circle (0.32cm);
\draw(0.5,1) circle (0.32cm);
\draw(3,3) circle (0.32cm);
\draw(3,1) circle (0.32cm);
\draw [->] (0.86,3) -- (2.64,3);
\draw [->] (3,2.68) -- (3,1.33);
\draw [->] (1.4,1) -- (0.86,1);
\draw [->] (0.5,2.04) -- (0.5,2.6);
\draw (0.15,3.3) node[anchor=north west] {$B_n$};
\draw (2.65,3.3) node[anchor=north west] {$B_1$};
\draw (2.65,1.3) node[anchor=north west] {$B_2$};
\draw (1.3,3.55) node[anchor=north west] {$\alpha_n$};
\draw (3,2.15) node[anchor=north west] {$\alpha_1$};
\draw (1.8,1.2) node[anchor=north west] {$\cdots$};
\draw (0.3,2.1) node[anchor=north west] {$\vdots$};
\draw (2.5,2.7) node[anchor=north west] {$+$};
\end{tikzpicture}
\caption{Module $\scrm$ for $NC_{\widetilde{A_n}}({\bf d})$}\label{Fig4}
\end{figure}

\item For the remaining cases in \cite[Tables I, II]{Mal}, we appeal to
Figure 3.3 as in Case $11$ above, with three suitably chosen generators
in each case.
\end{enumerate}\smallskip

\noindent \textit{Case $15$.}
Finally, we consider the remaining ``nongenuine, crystallographic'' cases
as in \cite[Table 2]{Po1}. Thus, $E/W$ is compact and $\Lin(W)$ is the
complexification of a real reflection group. In these cases, verify by
inspection from \cite[Table 2]{Po1} that the cocycle $c$ is always
trivial. Thus $W = \Lin(W) \ltimes \Tran(W)$, with $W' := \Lin(W)$ a
finite Weyl group, and $\Tran(W)$ a lattice of rank $2 |I'|$, where $I'$
indexes the simple reflections in the Weyl group $W'$.

We now \textbf{claim} that the corresponding family of generalized
nil-Coxeter algebras $NC_W({\bf d})$ are not finitely generated as
$\bk$-modules.
To show the claim requires a presentation of $W$ in terms of generating
reflections. The following recipe for such a presentation was
communicated to us by Popov \cite{Po2}. Notice from e.g. \cite[Table
2]{Po1} that $\Tran(W)$ is a direct sum of two $\Lin(W)$-stable lattices
$\Lambda_1$ and $\Lambda_2 = \alpha \Lambda_1$ (with $\alpha \not\in
\mathbb{R}$), each of rank $|I'|$. Thus, $\Lambda_1 \cong \Lambda_2$ as
$\mathbb{Z} W'$-modules, with $W' = \Lin(W)$ a finite real reflection
group as above. Moreover, for $j=1,2$, the semidirect product $S_j := W'
\ltimes \Lambda_j$ is a real crystallographic reflection group whose
fundamental domain is a simplex; this yields a presentation of $S_j$ via
$|I'|+1$ generating reflections in the codimension-one faces of this
simplex. One now combines these presentations for $S_1,S_2$ to obtain a
system of $|I'|+2$ generators for $W$; see in this context the remarks
following \cite[Theorem 3.1]{Mal}.
In this setting, it follows by \cite[Theorem 4.5]{Po1} that each $S_j$ is
isomorphic, as a real reflection group, to the affine Weyl group
$\widetilde{W'}$ over $W'$, since the Coxeter type of $S_j$ is determined
by the Coxeter types of $W'$ and $\Lambda_j$. Thus $W$ is in some sense a
`double affine Weyl group'.
(For simply-laced $W'$, it is also easy to verify by inspection from
\cite[Table 2]{Po1} that $\Lambda_j$ is isomorphic as a $\mathbb{Z}
W'$-module to the root lattice for $W'$, whence $S_j \cong
\widetilde{W'}$ for $j=1,2$.)

Equipped with this presentation of $W$ from \cite{Po2}, we analyze
$NC_W({\bf d})$ as follows. Fix a $\mathbb{Z} W'$-module isomorphism
$\varphi : \Lambda_1 \to \Lambda_2$, and choose affine reflections
$s_{0j} \in S_j$, corresponding to $\mu_1$ and $\mu_2 = \varphi(\mu_1)$
respectively, which together with the simple reflections $\{ s_i : i \in
I' \} \subset W$ generate $S_j \cong \widetilde{W'}$.
Then $W \twoheadrightarrow \widetilde{W'}$ upon quotienting by the
relation: $s_{01} = s_{02}$.
Using the presentation of $NC_W({\bf d})$ via the corresponding $|I'|+2$
generators $\{ T_i : i \in I' \} \sqcup \{ T_{01}, T_{02} \}$,
\[
NC_W({\bf d}) \twoheadrightarrow NC_W((2,\dots,2)) \twoheadrightarrow
NC_W((2,\dots,2)) / (T_{01} - T_{02}) \cong
NC_{\widetilde{W'}}((2,\dots,2)),
\]

\noindent and this last term is an affine Weyl nil-Coxeter algebra, hence
is not finitely generated as a $\bk$-module. Therefore neither is
$NC_W({\bf d})$, as desired.

This shows that $(1) \implies (2)$; the converse follows by \cite[Chapter
7]{Hum} and Theorem \ref{ThmA}.\smallskip

We now show that $(2)$ and $(3)$ are equivalent. Note from the above
case-by-case analysis that if $NC_W({\bf d})$ is not finitely generated,
then either it surjects onto an affine Weyl nil-Coxeter algebra
$NC_{\widetilde{W'}}((2,\dots,2))$, or one can define a module $\scrm$ as
above, and for each $r \geqslant 1$ there exists a word $\T_{w_r} \in
NC_W({\bf d})$, expressed using $O(r)$ generators, which sends the
$\bk$-basis vector $A_1 \in \scrm$ to $A_r$.
It follows in both cases that $\m$ is not nilpotent. Next, if $W = W(M)$
is a finite Coxeter group, then it is well-known (see e.g. \cite[Chapter
7]{Hum}) that $\m$ is nilpotent. Finally, if $NC(M) = NC_A(n,d)$, then
$\m$ is nilpotent by Theorem \ref{ThmB}. This shows $(2)
\Longleftrightarrow (3)$.

The final statement on the length function $\ell$ and the longest element
also follows from \cite[Chapter 7]{Hum} and Theorem
\ref{ThmB}.\qed\smallskip

\begin{remark}
If $M$ is a generalized Coxeter matrix with some $m_{ij} = \infty$, then
we can similarly work with $\scrm$ the $\bk$-span of $\{ a_r, b_r : r
\geqslant 1 \}$, where
\[
T_i : a_r \mapsto b_r, \ b_r \mapsto 0, \qquad
T_j : b_r \mapsto a_{r+1}, \ a_r \mapsto 0,
\]
and all other $T_k$ kill $\scrm$. It follows that $NC(M)$ once again has
infinite $\bk$-rank.
\end{remark}

\section*{Acknowledgments}

The author is grateful to Ivan Marin, Vladimir Popov, and Victor Reiner
for informative and stimulating correspondences; and to James Humphreys
for carefully going through an earlier draft and providing valuable
feedback. The author also thanks Daniel Bump, Gunter Malle, Eric Rowell,
Travis Scrimshaw, and Bruce Westbury for useful references and
discussions. The author is partially supported by a Young Investigator
Award from the Infosys Foundation.




\begin{thebibliography}{BMR2}
\bibitem[AnSc]{AnSc}
Nicol\'as Andruskiewitsch and Hans-J\"urgen Schneider, {\em On the
  classification of finite-dimensional pointed Hopf algebras},
  \href{http://dx.doi.org/10.4007/annals.2010.171.375}{Annals of
  Mathematics} 171(1):375--417, 2010.

\bibitem[As]{As}
Joachim Assion, {\em A proof of a theorem of Coxeter},
  C.~R.~Math.~Rep.~Acad.~Sci.~Canada 1(1):41--44, 1978/79.

\bibitem[Ba]{Bas1}
Tathagata Basak, {\em On Coxeter diagrams of complex reflection groups},
  \href{https://dx.doi.org/10.1090/S0002-9947-2012-05517-6 }{Transactions
  of the American Mathematical Society} 364(9):4909--4936, 2012.

\bibitem[BSS]{BSS}
Chris Berg, Franco Saliola, and Luis Serrano, {\em Pieri operators on the
  affine nilCoxeter algebra},
  \href{https://dx.doi.org/10.1090/S0002-9947-2013-05895-3}{Transactions
  of the American Mathematical Society} 366(1):531--546, 2014.

\bibitem[BGG]{BGG}
Joseph N.~Bernstein, Israel M.~Gelfand, and Sergei I.~Gelfand, {\em 
  Schubert cells and cohomology of the spaces $G/P$},
  \href{http://iopscience.iop.org/article/10.1070/RM1973v028n03ABEH001557/}{Russian
  Mathematical Surveys} 28(2):1--26, 1973.

\bibitem[BS]{BS}
Joseph N.~Bernstein and Ossip Schwarzman, {\em Complex crystallographic
  Coxeter groups and affine root systems},
  \href{http://dx.doi.org/10.2991/jnmp.2006.13.2.2}{Journal of Nonlinear
  Mathematical Physics} 13(2):163--182, 2006.

\bibitem[Be]{Be}
David Bessis, {\em Zariski theorems and diagrams for braid groups},
  \href{http://dx.doi.org/10.1007/s002220100155}{Inventiones
  mathematicae} 145(3):487--507, 2001.


\bibitem[BMR1]{BMR1}
Michel Brou\'e, Gunter Malle, and Rapha\"el Rouquier, {\em On complex
  reflection groups and their associated braid groups},
  \href{http://bookstore.ams.org/CMSAMS-16}{Representations of groups}
  (Banff, 1994), CMS Conference Proceedings 16:1--13, American
  Mathematical Society, Providence, 1995.

\bibitem[BMR2]{BMR2}
Michel Brou\'e, Gunter Malle, and Rapha\"el Rouquier, {\em Complex
  reflection groups, braid groups, Hecke algebras},
  \href{https://dx.doi.org/10.1515/crll.1998.064}{Journal f\"ur die reine
  und angewandte Mathematik} 500:127--190, 1998.

\bibitem[Ca]{Ca}
Roger W.~Carter, {\em Representation theory of the $0$-Hecke algebra},
  \href{http://dx.doi.org/10.1016/0021-8693(86)90238-3}{Journal
  of Algebra} 104(1):89--103, 1986.

\bibitem[Ch]{Ch}
Eirini Chavli, {\em The Brou\'e--Malle--Rouquier conjecture for the
  exceptional groups of rank 2},
  \href{https://arxiv.org/abs/1608.00834}{Ph.D.~thesis}, 2016.

\bibitem[Coh]{Coh}
Arjeh M.~Cohen, {\em Finite complex reflection groups},
  \href{http://www.numdam.org/item?id=ASENS_1976_4_9_3_379_0}{Annales
  scientifiques de l'\'Ecole Normale Sup\'erieure}, S\'er.~4,
  9(3):379--436, 1976.

\bibitem[Cox1]{Cox}
Harold S.M.~Coxeter, {\em Discrete groups generated by reflections},
  \href{http://www.jstor.org/stable/1968753}{Annals of Mathematics}
  35(3):588--621, 1934.

\bibitem[Cox2]{Cox2}
Harold S.M.~Coxeter, {\em The complete enumeration of finite groups of
  the form $R_i^2 = (R_i R_j)^{k_{ij}} = 1$},
  \href{http://dx.doi.org/10.1112/jlms/s1-10.37.21}{Journal of the London
  Mathematical Society} s1-10(1):21--25, 1935.

\bibitem[Cox3]{Co}
Harold S.M.~Coxeter, {\em Factor groups of the braid group}, Proceedings
  of the 4th Canadian Mathematical Congress (Banff, 1957), University of
  Toronto Press, 95--122, 1959.

\bibitem[Dr]{Dr}
Vladimir G.~Drinfeld, {\em Degenerate affine Hecke algebras and
  Yangians}, \href{http://dx.doi.org/10.1007/BF01077318}{Functional
  Analysis and its Applications} 20(1):58--60, 1986.

\bibitem[Et]{Et}
Pavel Etingof, {\em Proof of the Brou\'e--Malle--Rouquier conjecture in
  characteristic zero (after I. Losev and I. Marin--G. Pfeiffer)},
  \href{http://dx.doi.org/10.1007/s40598-017-0069-7}{Arnold Mathematical
  Journal} 3(3):445--449, 2017.

\bibitem[EG]{EG}
Pavel Etingof and Victor Ginzburg, {\em Symplectic reflection algebras,
  Calogero-Moser space, and deformed Harish-Chandra homomorphism},
  \href{http://dx.doi.org/10.1007/s002220100171}{Inventiones
  mathematicae} 147(2):243--348, 2002.

\bibitem[ES]{ES}
Pavel Etingof and Olivier Schiffmann, {\em Lectures on Quantum Groups},
  Lectures in Mathematical Physics, International Press, Somerville, MA,
  1998.


\bibitem[FS]{FS}
Sergey Fomin and Richard Stanley, {\em Schubert polynomials and the
  nil-Coxeter algebra},
  \href{http://dx.doi.org/10.1006/aima.1994.1009}{Advances in
  Mathematics} 103(2):196--207, 1994.


\bibitem[GHV]{GHV}
Mat\'ias Gra\~na, Istv\'an Heckenberger, and Leandro Vendramin, {\em
  Nichols algebras of group type with many quadratic relations},
  \href{http://dx.doi.org/10.1016/j.aim.2011.04.006}{Advances in
  Mathematics} 227(5):1956--1989, 2011.


\bibitem[HV1]{HV1}
Istv\'an Heckenberger and Leandro Vendramin, {\em A classification of
  Nichols algebras of semi-simple Yetter-Drinfeld modules over
  non-abelian groups}, 
  \href{http://dx.doi.org/10.4171/JEMS/667}{Journal of the European
  Mathematical Society} 19(2):299--356, 2017.

\bibitem[HV2]{HV2}
Istv\'an Heckenberger and Leandro Vendramin, {\em The classification of
  Nichols algebras over groups with finite root system of rank two},
  \href{http://dx.doi.org/10.4171/JEMS/711}{Journal of the European
  Mathematical Society} 19(7):1977--2017, 2017.

%

\bibitem[Hug1]{Hug1}
Mervyn C.~Hughes, {\em Complex reflection groups},
  \href{http://dx.doi.org/10.1080/00927879008824120}{Communications in
  Algebra} 18(12):3999--4029, 1990.

\bibitem[Hug2]{Hug2}
Mervyn C.~Hughes, {\em Extended root graphs for complex reflection
  groups},
  \href{http://dx.doi.org/10.1080/00927879908826424}{Communications
  in Algebra} 27(1):119--148, 1999.

\bibitem[Hum]{Hum}
James E.~Humphreys, {\em Reflection groups and Coxeter groups},
  Cambridge studies in advanced mathematics no. \textbf{29}, Cambridge
  University Press, Cambridge-New York-Melbourne, 1990.


\bibitem[Kha]{Kh}
Apoorva Khare, {\em Generalized nil-Coxeter algebras, cocommutative
  algebras, and the PBW property},
  \href{http://dx.doi.org/10.1090/conm/688/13832}{AMS Contemporary
  Mathematics} 688:139--168, 2017.

\bibitem[Kho]{Kho}
Mikhail Khovanov, {\em NilCoxeter algebras categorify the Weyl algebra},
  \href{http://www.tandfonline.com/doi/abs/10.1081/AGB-100106800}{Communications
  in Algebra} 29(11):5033--5052, 2001.

\bibitem[KL]{KL}
Mikhail Khovanov and Aaron D.~Lauda, {\em A diagrammatic approach to
  categorification of quantum groups I},
  \href{http://dx.doi.org/10.1090/S1088-4165-09-00346-X}{Representation
  Theory} 13:309--347, 2009.

\bibitem[KK]{KK}
Bertram Kostant and Shrawan Kumar, {\em The nil Hecke ring and cohomology
  of $G/P$ for a Kac-Moody group $G$},
  \href{http://dx.doi.org/10.1016/0001-8708(86)90101-5}{Advances in
  Mathematics} 62(3):187--237, 1986.

\bibitem[Ko]{Ko}
David W.~Koster, {\em Complex reflection groups}, Ph.D. thesis,
  University of Wisconsin, 1975.

\bibitem[LS]{LS}
Alain Lascoux and Marcel-Paul Sch\"utzenberger, {\em Fonctorialit\'e des
  polyn\^omes de Schubert},
  \href{http://www.ams.org/books/conm/088/}{Invariant Theory} (R. Fossum,
  W. Haboush, M. Hochster and V. Lakshmibai, Eds.), Contemporary
  Mathematics Volume 88, pages 585--598, American Mathematical Society,
  Providence, 1989.

\bibitem[LT]{LT}
Gustav I.~Lehrer and Donld E.~Taylor, {\em Unitary reflection groups},
  Australian Mathematical Society Lecture Series Volume 20, Cambridge
  University Press, Cambridge, 2009.

\bibitem[Lo]{Lo}
Ivan Losev, {\em Finite-dimensional quotients of Hecke algebras},
  \href{http://dx.doi.org/10.2140/ant.2015.9.493}{Algebra \& Number
  Theory} 9(2):493--502, 2015.

\bibitem[Lu]{Lu}
George Lusztig, {\em Affine Hecke algebras and their graded version},
  \href{http://dx.doi.org/10.1090/S0894-0347-1989-0991016-9}{Journal of
  the American Mathematical Society} 2:599--635, 1989.

\bibitem[Mal]{Mal}
Gunter Malle, {\em Presentations for crystallographic complex reflection
  groups}, \href{http://dx.doi.org/10.1007/BF02549209}{Transformation
  Groups} 1(3):259--277, 1996.

\bibitem[Mar]{Ma}
Ivan Marin, {\em The freeness conjecture for Hecke algebras of complex
  reflection groups, and the case of the Hessian group $G_{26}$},
  \href{http://dx.doi.org/10.1016/j.jpaa.2013.08.009}{Journal of Pure and
  Applied Algebra} 218(4):704--720, 2014.

\bibitem[MP]{MP}
Ivan Marin and G\"otz Pfeiffer, {\em The BMR freeness conjecture for the
  2-reflection groups},
  \href{https://dx.doi.org/10.1090/mcom/3142}{Mathematics of
  Computation}, published online, 2016.

\bibitem[No]{Nor}
P.N.~Norton, $0$-{\em Hecke algebras},
  \href{http://dx.doi.org/10.1017/S1446788700012453}{Journal of the
  Australian Mathematical Society A}, 27(3):337--357, 1979.


\bibitem[ORS]{ORS}
Peter Orlik, Victor Reiner, and Anne V.~Shepler, {\em The sign
  representation for Shephard groups},
  \href{http://dx.doi.org/10.1007/s002080200001}{Mathematische Annalen}
  322(3):477--492, 2002.


\bibitem[Pop1]{Po1}
Vladimir L.~Popov, {\em Discrete complex reflection groups},
  \href{http://dx.doi.org/10.13140/2.1.4050.2088}{Communications of the
  Mathematical Institute}, vol.~15, Rijksuniversiteit Utrecht, 89 pp.,
  1982.

\bibitem[Pop2]{Po2}
Vladimir L.~Popov, personal communication, Feb.~2016.


\bibitem[ReS]{ReS}
Victor Reiner and Anne V.~Shepler, {\em Invariant derivations and
  differential forms for reflection groups}, preprint,
  \href{http://arxiv.org/abs/1612.01031}{arXiv:1612.01031}.

\bibitem[ST]{ST}
Geoffrey C.~Shephard and John A.~Todd, {\em Finite unitary reflection
  groups}, \href{http://dx.doi.org/10.4153/CJM-1954-028-3}{Canadian
  Journal of Mathematics} 6:274--304, 1954.

\bibitem[SW1]{SW2}
Anne V.~Shepler and Sarah Witherspoon, {\em A Poincar\'e--Birkhoff--Witt
  theorem for quadratic algebras with group actions},
  \href{http://www.ams.org/journals/tran/2014-366-12/S0002-9947-2014-06118-7/}{Transactions
  of the American Mathematical Society} 366(12):6483--6506, 2014.

\bibitem[SW2]{SW4}
Anne V.~Shepler and Sarah Witherspoon, {\em Poincar\'e--Birkhoff--Witt
  Theorems}, Commutative Algebra and Noncommutative Algebraic Geometry,
  \href{http://www.cambridge.org/yt/academic/subjects/mathematics/algebra/commutative-algebra-and-noncommutative-algebraic-geometry-volume-1}{Volume
  I: Expository Articles} (D. Eisenbud, S.B. Iyengar, A.K. Singh, J.T.
  Stafford, and M. Van den Bergh, Eds.),
  Mathematical Sciences Research Institute Proceedings Volume 67, pages
  259--290, Cambridge University Press, Cambridge, 2015.

%
\end{thebibliography}
\end{document}